%% file: barycenter-langevin.tex
\documentclass[11pt]{article}
\usepackage[utf8]{inputenc}
\usepackage[margin=1.1in,a4paper]{geometry}
\usepackage[shortlabels]{enumitem}
\usepackage{mathtools}
\usepackage{amsmath, amssymb, graphicx, url, amsthm, subcaption, dsfont}
\usepackage{mathrsfs}
\usepackage[ruled, lined]{algorithm2e}

\usepackage[maxbibnames=99, style=alphabetic, maxalphanames=5]{biblatex}
\usepackage[mathcal]{eucal}
\usepackage{tikz}
\usepackage{verbatim}
 \usepackage{tcolorbox}
 \newtcolorbox{assbox}{colback=black!5!white,colframe=black!75!black}
  \newtcolorbox{thmbox}{colback=red!5!white,colframe=red!75!black}

\usepackage{amsfonts}       
\usepackage{enumitem}
\usepackage{microtype}      
\usepackage{xcolor}
\usepackage[colorlinks=true,citecolor=blue]{hyperref}    
\usepackage{diagbox}   
\usepackage{tcolorbox}
\usepackage{nicefrac}

\newcommand{\rev}[1]{{#1}}

\input{shortcuts}

\graphicspath{ {./images/} } 

\title{
Doubly Regularized Entropic Wasserstein Barycenter\footnote{To appear in~\emph{Foundations of Computational Mathematics}}
}

\author{
L\'ena\"ic Chizat\thanks{Institut de Mathématiques, École polytechnique fédérale de Lausanne (EPFL), \texttt{lenaic.chizat@epfl.ch}} 
}

\begin{document}
\maketitle

\begin{abstract}
We study a general formulation of regularized Wasserstein barycenters that enjoy favorable regularity, approximation, stability and (grid-free) optimization properties.
This barycenter is defined as the unique probability measure that minimizes the sum of entropic optimal transport (EOT) costs with respect to a family of given probability measures, plus an entropy term. We denote it the $(\lambda,\tau)$-barycenter, where $\lambda$ is the inner regularization strength and $\tau$ the outer one. This formulation recovers several previously proposed EOT barycenters for various choices of $\lambda,\tau \geq 0$ and generalizes them.
\rev{First, we show that, as $\lambda, \tau \to 0$, regularizing doubly can \emph{decrease} the approximation error compared to a single regularization. More specifically, we show that for smooth densities and the quadratic cost, the leading order term of the suboptimality in the (unregularized) Wasserstein barycenter objective cancels when $\tau\sim \nicefrac{\lambda}{2}$.} We discuss also this phenomenon for isotropic Gaussian distributions where all $(\lambda,\tau)$-barycenters have closed-form.
Second, we show that for $\lambda,\tau>0$, this barycenter has a smooth density and is strongly stable under perturbation of the marginals. In particular, it can be estimated efficiently: given $n$ samples from each of the probability measures, it converges in relative entropy to the population barycenter at a rate $n^{-1/2}$.
Finally, this formulation is amenable to a grid-free optimization algorithm: we propose a simple Noisy Particle Gradient Descent method which, in the mean-field limit, converges globally at an exponential rate to the $(\lambda,\tau)$-barycenter. 
\end{abstract}

\section{Introduction}
The Wasserstein barycenter is a probability measure that summarizes a family of probability measures in a geometrically meaningful way. This object, first studied in~\cite{agueh2011barycenters}, has found numerous applications in statistics~\cite{bassetti2006minimum, boissard2015distribution, bernton2019parameter}, image processing~\cite{rabin2011wasserstein}, computer graphics~\cite{solomon2015convolutional} and Bayesian inference~\cite{srivastava2018scalable, backhoff2022bayesian} (see~\cite[Chap.~9.2]{peyre2019computational} or~\cite{panaretos2020invitation} for introductions to this topic). While it is arguably one of the most useful notion of barycenters for probability measures, the Wasserstein barycenter is unfortunately difficult to estimate and compute in large-scale applications.

To overcome these limitations and following the idea of entropic regularization of optimal transport, a.k.a~the Schr\"odinger bridge problem~\cite{schrodinger1932theorie,wilson1969use,erlander1990gravity,kosowsky1994invisible,leonard2012schrodinger,cuturi2013sinkhorn}, various formulations of entropy-regularized Wasserstein barycenters have been proposed and studied in the literature (discussed below). In these works, entropic regularization is incorporated in two different ways, which we refer to as \emph{inner} and \emph{outer} regularizations, and with various reference measures. It is not clear a priori how these formulations relate to each other, and whether a particular one stands out for its mathematical or practical properties.

In this paper, we aim at clarifying and generalizing the picture by considering \emph{both} inner and outer regularizations at the same time. We prove that this formulation combines favorable analytical, approximation, stability and optimization properties, which are all enabled by the joint contributions of these two regularizations.

\subsection{Entropic optimal transport (EOT)}
Let $\Xx$ be a compact and convex subset of $\RR^d$ with nonempty interior and let $c\in \Cc^p(\Xx\times\Xx)$ for some $p\geq 2$.
For two probability measures $\mu, \nu \in \Pp(\Xx)$, let $\Pi(\mu,\nu)$ be the set of transport plans\footnote{That is, probability measures on $\Xx\times \Xx$ with marginals $\mu$ and $\nu$ on each factor of $\Xx\times \Xx$.} between $\mu$ and $\nu$ and define the Entropic Optimal Transport (EOT) cost   as
\begin{align}\label{eq:EOT}
T_\lambda(\mu,\nu) \coloneqq \min_{\gamma \in \Pi(\mu,\nu)} \int_{\Xx\times \Xx} c(x,y) \d\gamma(x,y) + \lambda \KL(\gamma | \mu\otimes \nu) 
\end{align}
where $\lambda\geq 0$ is the regularization strength and $\KL(\mu|\nu)=\int \log\big(\frac{\d\mu}{\d\nu}\big)\d\mu$ if $\mu\ll\nu$ and $+\infty$ otherwise, is the relative entropy (or Kullback-Leibler divergence). Notice the choice of reference measure $\mu\otimes \nu$ for the regularization term in~\eqref{eq:EOT}, which is important for our exposition, see Section~\ref{sec:reference-measure}. With this choice, $T_\lambda$ is always finite, even for discrete measures (indeed, $\gamma=\mu\otimes \nu$ is a transport plan). Setting the regularization $\lambda$ to $0$, we recover the standard optimal transport problem, and the $L^2$-Wasserstein distance $W_2$ is defined as $W_2\coloneqq \sqrt{2T_0}$ when $c(x,y)=\frac12 \Vert y-x\Vert^2_2$.

\subsection{Doubly Regularized EOT Barycenter}\label{sec:intro-douba}

Given a family of $K$ probability measures $(\nu_1,\dots,\nu_K)\in \Pp(\Xx)^K$ and $K$ weights $(w_1,\dots,w_K)\in \RR_+^K$ summing to $1$, we define the EOT barycenter functional $G_\lambda:\Pp(\Xx)\to \RR$ as
\begin{align}\label{eq:G}
G_\lambda(\mu) \coloneqq \sum_{k=1}^K w_kT_\lambda(\mu,\nu_k).
\end{align}
We also define the $H$-functional of Boltzmann (the opposite of the differential entropy) by 
\begin{align}\label{eq:entropy}
H(\mu) \coloneqq  \int_\Xx \log\Big(\frac{\d \mu}{\d x}\Big)\d \mu(x)
\end{align}
if $\mu$ is absolutely continuous and $+\infty$ otherwise and  $\d x$ stands for the Lebesgue measure. Our main object of study is the following doubly regularized EOT barycenter.

\begin{assbox}
\begin{definition}\label{def:debiaised-barycenter} For $\tau,\lambda\geq 0$, the \emph{$(\lambda,\tau)$-barycenter} is the minimizer $\mu^*_{\lambda,\tau} \in \Pp(\Xx)$ of
\begin{align}\label{eq:optim-objective}
F_{\lambda,\tau}(\mu) \coloneqq G_\lambda(\mu) + \tau H(\mu).
\end{align}
We refer to $\lambda$ (resp.~$\tau$) as the \emph{inner} (resp.~\emph{outer}) regularization strength.
\end{definition}
\end{assbox}

For absolutely continuous measures, $(\lambda,\tau)$-barycenters can also be interpreted as EOT barycenters with only inner regularization, but with a different reference measure $\sigma_{\mathrm{ref}}$ in place of $\mu\otimes \nu$ in the definition of $T_\lambda$ of Eq.~\eqref{eq:EOT}. Specifically, as shown in Section~\ref{sec:reference-measure}, the $(\lambda,\tau)$-barycenter coincides with the $(\lambda,0)$-barycenter with reference measure
\begin{align}
\sigma_{\mathrm{ref}} = \Big[\Big(\frac{\d\mu}{\d x}\Big)^{\alpha}\d x\Big]\otimes \Big[\Big(\frac{\d\nu}{\d y}\Big)^{\alpha}\d y\Big]&&\text{where}&&\alpha \coloneqq 1-\tau/\lambda<1.
\end{align} 
This in particular includes the case $\alpha=0$, that is $\sigma_{\mathrm{ref}}=\d x \otimes \d y$ (when $\tau=\lambda$) which is the historical formulation of EOT with Lebesgue as a reference measure known as the Schr\"odinger bridge problem. Under this equivalent formulation, $(\lambda,\lambda)$-barycenters have been considered and studied in many works~\cite{cuturi2014fast,cuturi2018semidual,bigot2019penalization} (see, e.g.~the reference book~\cite[Chap.~9.2]{peyre2019computational}). As we will see, interpreting those barycenters as doubly-regularized lead to a streamlined and stronger analysis of their properties. We will also see that $(\lambda,\nicefrac\lambda2)$-barycenters (corresponding to $\alpha=\nicefrac12$) approximate better the $(0,0)$-barycenter as $\lambda\to 0$. Let us mention that the general formulation of $(\lambda,\tau)$-barycenters in Def.~\ref{def:debiaised-barycenter} already appears in~\cite{ballu2020stochastic} where it is motivated by computational purposes and the analysis in that paper requires the constraint $\tau>\lambda$. \rev{In contrast, we are mostly interested in the regime $\tau \leq \lambda$ since we show in Section~\ref{sec:approximation} that this regime sometimes lead to improved approximation properties.}

\subsection{Relation to barycenters in the literature}
Various notions of barycenters for probability measures based on (entropy-regularized) optimal transport can be found in the literature. Most of them can be seen as $(\lambda,\tau)$-barycenters for particular choices of $\lambda$ and $\tau$, see Table~\ref{tab:comparison} for a summary.
\begin{enumerate}[(i)]
\item \textbf{Unregularized OT barycenters}. They are the minimizers $\mu_{0,0}^*$ of $G_0$ in $\Pp(\Xx)$, they were first studied in~\cite{agueh2011barycenters} and correspond to $(0,0)$-barycenters in Def.~\ref{def:debiaised-barycenter}. Under our assumptions, minimizers always exist but might not be unique. Uniqueness holds for instance in the case $c(x,y)\propto \Vert y-x\Vert_2^2$ and if at least one of the $\nu_k$ vanishes on small-sets (a condition weaker than absolute continuity)~\cite{agueh2011barycenters}. 
\item \textbf{Inner-regularized barycenters}. They are defined as the minimizers of $G_\lambda$ and correspond to $(\lambda,0)$-barycenters. Uniqueness of this barycenter is not always granted as $G_\lambda$ is not strictly convex (think of $c$ the constant cost for which any $\mu\in\Pp(\Xx)$ is a barycenter). Compared to $\mu_{0,0}^*$, the regularization typically induces a shrinking bias. In particular, the barycenter of Gaussian distributions can be a Dirac mass when their covariances is small compared to $\lambda$~\cite{janati2020debiased} (see  Section~\ref{sec:gaussians}). To understand this, one may verify that for the square-distance cost, the minimizer of $T_\lambda(\cdot,\mu)$ is $\mu$ itself only when $\lambda=0$ or when $\mu$ is a Dirac mass, and otherwise, it is a deconvolution of $\mu$~\cite{rigollet2018entropic}. 
\item \textbf{Sinkhorn divergence barycenters.} The Sinkhorn divergence $
 S_{\lambda}(\mu,\nu) \coloneqq T_\lambda(\mu,\nu) - \frac12 T_\lambda(\mu,\mu)-\frac12 T_\lambda(\nu,\nu)
 $
has been introduced with the motivation to fix this bias and recover a distance-like quantity~\cite{ramdas2017wasserstein, genevay2018learning}. It is indeed a positive definite quantity as long as $e^{-c/\lambda}$ is a positive definite universal kernel~\cite{feydy2019interpolating}, which is the case e.g. for $c(x,y)\propto\Vert x-y\Vert^2_2$. This suggests to consider \emph{Sinkhorn divergence} barycenters~\cite{janati2020debiased}, i.e.~the minimizers of 
 $$
\mu \mapsto \sum_{k=1}^K w_k S_\lambda(\mu,\nu_k) = G_\lambda(\mu) - \tfrac12 T_{\lambda}(\mu,\mu) + C
 $$
 where $C\in \RR$ does not depend on $\mu$. We will use the notation $\mu_{\lambda,\mathrm{div}}^*$ for this barycenter.
 The self-EOT term indeed reduces the bias of the barycenter: this is shown for Gaussian distributions in~\cite{janati2020debiased}, and for general smooth measures in Section~\ref{sec:approximation}. However, little else is known about this barycenter regarding uniqueness, stability or regularity. This formulation is not covered by Def.~\ref{def:debiaised-barycenter}.
\item \textbf{Schr\"odinger barycenters\footnote{These were referred to as \emph{Sinkhorn barycenters} in~\cite{bigot2019penalization}; here we propose a name that conveys the choice of reference measure in the formulation.
}.} In most works (e.g.~\cite{cuturi2014fast,cuturi2018semidual,bigot2019penalization}), inner-regularized barycenters are in fact considered with Lebesgue product measure $\d x\otimes \d x$ as a reference measure in~\eqref{eq:EOT}, in place of $\mu\otimes \nu$. As discussed in the previous paragraph, they correspond to $(\lambda,\lambda)$-barycenters. This regularization leads to a blurring bias.
\item  \textbf{Outer-regularized barycenters.} These are the minimizers of  $G_0+\tau H$, studied in~\cite{bigot2019data, carlier2021entropic}, which correspond to $(0,\tau)$-barycenters in Def.~\ref{def:debiaised-barycenter}. This barycenter has interesting regularity properties: for instance~\cite{carlier2021entropic} show bounds on the $L^\infty$ norm, moments and regularity of the barycenter, which are not known for the inner-regularized barycenters. This regularization induces a blurring bias as well.
\end{enumerate}
As illustrated by this list, all the previously proposed OT-like barycenters -- except the Sinkhorn divergence barycenter -- appear to be $(\lambda,\tau)$-barycenters, with various formulations corresponding to different subsets of the $(\lambda,\tau)$ plane. In our analysis, we will often restrict ourselves to $\lambda,\tau>0$ for convenience.

\begin{table}
\centering
\begin{tabular}{l|l|l}
Barycenter & Objective  & Notation  \\
\hline
Un-regularized & $G_0$ &  $\mu_{0,0}^*$   \\
Inner-regularized &$G_\lambda$ & $\mu_{\lambda,0}^*$   \\
Sinkhorn divergence & $G_\lambda -\frac12 T_\lambda(\cdot,\cdot)$  &  $\mu_{\lambda,\mathrm{div}}^*$  \\
Schr\"odinger & $G_\lambda + \lambda H$ & $\mu^*_{\lambda,\lambda}$\\
Outer-regularized &$G_0 +\tau H$ &  $\mu_{0,\tau}^*$   \\
\textbf{Doubly-regularized} &$G_\lambda +\tau H$  & $\mu_{\lambda,\tau}^*$    \\
\hline
\end{tabular}
\caption{List of the various formulations of OT-like barycenters. 
See Fig.~\ref{fig:comparison-1D} for an illustration.
}\label{tab:comparison}
\end{table}

\subsection{Contributions}
The contributions of our work are the following:
\begin{itemize}
\item in Section~\ref{sec:statics}, we discuss basic variational properties of $(\lambda,\tau)$-barycenters. In particular, we prove that for $\lambda,\tau>0$, they have a smooth log-density (Thm.~\ref{thm:characterization}) and we derive a dual formulation (Prop.~\ref{prop:dual}).
\item in Section~\ref{sec:approximation}, we study the approximation error of $(\lambda,\tau)$-barycenter with respect to the unregularized Wasserstein barycenter for the square-distance cost. We prove that for smooth marginals $\nu_1,\dots,\nu_K$, the suboptimality of $(\lambda,\nicefrac{\lambda}{2})$-barycenters in the Wasserstein barycenter functional $G_0$ is of the order $\lambda^2$ and that the same holds for Sinkhorn divergence barycenters (Thm.~\ref{thm:approx}). We also compute and discuss the closed-form of $(\lambda,\tau)$-barycenters between isotropic Gaussian distributions (Prop.~\ref{prop:gaussians}). 

\item In Section~\ref{sec:statistics}, we give stability bounds for $(\lambda,\tau)$-barycenters (Thm.~\ref{thm:stability}). A consequence of these stability bounds is that they can be estimated in relative entropy given $n$ independent samples from each $\nu_k$ with an expected error in $O(\tau^{-1}(1+\lambda^{-d/2})n^{-1/2})$.

\item To compute this barycenter, we introduce in Section~\ref{sec:dynamics} a grid-free numerical method: Noisy Particle Gradient Descent (NPGD). We prove the well-posedness and exponential convergence to the global minimizer (Thm.~\ref{thm:global-convergence}) of this optimization dynamics in the mean-field limit, i.e. when the number of particles grows to infinity. 

\item Numerical results are presented in Section~\ref{sec:numerics}. There, we give examples of $(\lambda,\tau)$-barycenters on a simple 1D problem solved via convex optimization on the dual problem, and we illustrate the global convergence of the grid-free method NPGD on an example where $G_0$ has a spurious minimizer. 

\end{itemize}

\paragraph{\rev{General assumptions and notations}}
Throughout $\Xx \subset \RR^d$ is a compact convex set with nonempty interior and $c\in \Cc^p(\Xx\times \Xx)$ for some $p\in \NN,\; p\geq 2$. \rev{The space $\Cc^{p}(\Xx)$ is the space of functions defined on $\Xx$ that admit a $\Cc^p$ extension on $\RR^d$, endowed with the usual supremum norm. Using the multi-index notation, this norm is defined as
\(
\Vert f\Vert_{\Cc^k} \coloneqq \inf_{\tilde f}\sup_{\vert \alpha\vert \leq k} \Vert \tilde f^{(\alpha)}\Vert_\infty
\)
where the infimum is over functions $\tilde f$ that are extensions of $f$ defined on $\RR^d$. Endowed with this norm, $\Cc^p(\Xx)$ is a Banach space, see~\cite[Chap.~8, II]{queffelec2020analyse} for details. We say that a sequence $(\mu_n) \in\Pp(\Xx)^\NN$ converges \emph{weakly} to $\mu^*\in \Pp(\Xx)$ iff for any continuous function $f\in \Cc^0(\Xx)$, it holds $\lim_{n\to \infty} \int f\d\mu_n =\int f\d\mu^*$. A functional $E:\Pp(\Xx)\to \RR$ is said \emph{weakly continuous} iff for any weakly converging sequence $(\mu_n)$ it holds $\lim_n E(\mu_n)=E(\lim_n \mu_n)$.
}

\section{Well-posedness and regularity}\label{sec:statics}
This section contains basic mathematical results about $(\lambda,\tau)$-barycenters and the optimization problem defining them (Def.~\ref{def:debiaised-barycenter}). 

\subsection{Preliminaries: regularity of EOT} Let us begin with some useful facts about EOT. The problem~\eqref{eq:EOT} that defines $T_\lambda$ has a unique solution $\gamma^*\in \Pp(\Xx\times \Xx)$ and admits the dual formulation
\begin{align}\label{eq:dual-EOT}
T_\lambda(\mu,\nu) = \max_{\phi\in L^1(\mu),\psi\in L^1(\nu)} \int \phi \d\mu + \int \psi \d\nu +\lambda \Big(1 - \int e^{(\phi(x)+\psi(y)-c(x,y))/\lambda}\d\mu(x)\d\nu(y)\Big).
\end{align}
This dual problem has a solution $(\phi^*,\psi^*)$ which is unique in $L^1(\mu)\times L^1(\nu)$ up to the transformation $(\phi^*+c,\psi^*-c)$ for $c\in \RR$. At optimality, we have $T_\lambda(\mu,\nu) = \int \phi^*\d\mu + \int \psi^*\d\nu$ and the primal-dual relation
$
\gamma(\d x,\d y) = e^{(\phi(x)+\psi(y)-c(x,y))/\lambda}\mu(\d x)\nu(\d y).
$
Moreover, the potentials satisfy for $\mu\otimes \nu$ almost every $(x,y)$, the optimality condition
\begin{align}\label{eq:schrodinger-system}
\left\{
\begin{aligned}
\phi^*(x) &= -\lambda \log \Big( \int e^{(\psi^*(y)-c(x,y))/\lambda}\d\nu(y)\Big)\\
\psi^*(y) &= -\lambda \log \Big( \int e^{(\phi^*(x)-c(x,y))/\lambda}\d\mu(x)\Big)
\end{aligned}
\right. 
.
\end{align}
These equations can be used to extend $\phi^*$ and $\psi^*$ as continuous functions (in fact functions of class $\Cc^p$ when $c\in \Cc^p$) over $\Xx$, which satisfy these equations everywhere~\cite{genevay2019sample}. 

\begin{definition}[Schr\"odinger potentials]\label{def:potentials}
The pairs of functions $(\phi,\psi)\in \Cc^p(\Xx)\times \Cc^p(\Xx)$ which satisfy the \emph{Schr\"odinger system}\footnote{It would perhaps be less ambiguous to call these ``EOT'' system/potentials, as their regularity properties rely on using $\mu\otimes \nu$ as a reference measure in EOT instead of $\d x \otimes \d y$ in the original Schr\"odinger system.} Eq.~\eqref{eq:schrodinger-system} for all $(x,y)\in \Xx^2$ are called \emph{Schr\"odinger} potentials. This pair is unique up to the transformation $(\phi+C,\psi-C)$ for $C\in \RR$. \rev{In what follows, we fix an arbitrary $x_0\in \Xx$ and we call \emph{the} Schr\"odinger potentials, the unique pair $(\phi,\psi)$ that moreover satisfies $\phi(x_0)=0$.}
\end{definition}

This particular choice of potentials among all those that satisfy Eq.~\eqref{eq:schrodinger-system} $\mu$ (resp.~$\nu$) almost everywhere is justified by the following result.
\begin{proposition}[First-variation of entropic optimal transport]\label{prop:propertiesEOT} Fix $\nu \in \Pp(\Xx)$ and for $\mu\in \Pp(\Xx)$ let $(\phi[\mu],\psi[\mu])$ be the Schr\"odinger potentials associated to the pair $(\mu,\nu)$. Then:
\begin{itemize}
\item[(i)] The map $\mu\mapsto \phi[\mu]$ (as well as the map $\mu\mapsto \psi[\mu]$) satisfies the following Lipschitz continuity property: there exists $L>0$ such that
\[
\Vert \phi[\mu] - \phi[\mu']\Vert_{\Cc^{p-1}(\Xx)} \leq L\, W_2(\mu,\mu'),\quad \forall \mu,\mu'\in \Pp(\Xx).
\]
\item[(ii)]  the function 
$\mu\mapsto T_\lambda(\mu,\nu)$ is convex, weakly continuous, and admits $\phi[\mu]$ as first-variation, i.e.
\begin{align}
\forall \mu,\tilde \mu \in \Pp(\Xx),\; \lim_{\epsilon\, \downarrow\, 0}\frac{1}{\epsilon} \Big( T_\lambda((1-\epsilon)\mu+\epsilon \tilde \mu,\nu) - T_\lambda(\mu,\nu)\Big) = \int_\Xx \phi[\mu](x)\d (\tilde \mu-\mu)(x) .
\end{align}
\item[(iii)] For any $p'\in \{0,\dots,p\}$ there exists $C_{p'}>0$ independent of $\mu,\nu$ and $\lambda$ such that
\[
\Vert \phi[\mu]\Vert_{\Cc^{p'}} \leq C_{p'}\lambda^{\min\{0, 1-p'\}}.
\]

\end{itemize}
\end{proposition}
\begin{proof}
The first claim is technical and is proved in~\cite{carlier2024displacement} via the implicit function theorem on the Schr\"odinger system~\eqref{eq:schrodinger-system}. The convexity of $T_\lambda(\cdot,\nu)$ is clear by Eq.~\eqref{eq:dual-EOT} which expresses this function as a supremum of (weakly continuous) affine forms. To prove that $\phi[\mu]$ is the first-variation as in~\cite{feydy2019interpolating}, let $\mu_\epsilon \coloneqq (1-\epsilon)\mu+\epsilon\tilde \mu$. Using the fact that $\phi[\mu]$ (resp.~$\phi[\mu_\epsilon]$) is a subgradient of $T_\lambda(\cdot,\nu)$ at $\mu$ (resp. at $\mu_\epsilon$), we have
\begin{align*}
 \int \phi[\mu]\d(\tilde\mu-\mu)\leq \frac1\epsilon \Big(T_\lambda(\mu_\epsilon,\nu)-T_\lambda(\mu,\nu)\Big) \leq \int \phi[\mu_\epsilon]\d(\tilde \mu-\mu).
\end{align*}
The point \emph{(ii)} follows then from the weak continuity of $\mu\mapsto \phi[\mu]$, a consequence of \emph{(i)}. Finally \emph{(iii)} is proved in~\cite{genevay2019sample} where it is obtained by differentiating $p'$ times Eq.~\eqref{eq:schrodinger-system}  and applying Fa\`a di Bruno's formula.
\end{proof}
Let us mention that the Lipschitz constant $L$ in \emph{(i)} may depend exponentially on the oscillation of $c/\lambda$, namely $(\sup c -\inf c)/\lambda$. \rev{ See~\cite{divol2024tight} for an exponential improvement of this constant for the case of the cost $c(x,y)=\Vert y-x\Vert^2_2$.}

\subsection{Regularity of $(\lambda,\tau)$-barycenters}\label{sec:regularity}
We now gather useful regularity properties of $G_\lambda$, which are direct consequences of Prop.~\ref{prop:propertiesEOT}.
\begin{proposition}[Regularity of $G_\lambda$]\label{prop:Reg_G}
For any $(\nu_k)_{k=1}^K\in \Pp(\Xx)^K$, the function $G_\lambda \colon \mathcal P(\Xx) \to \RR_+$ defined in Eq.~\eqref{eq:G} is convex, weakly continuous, and for any $\mu \in \Pp(\Xx)$ it admits a first-variation
\begin{align}\label{eq:first-variation}
V[\mu] \coloneqq \sum_{k=1}^K w_k \phi_k[\mu] 
\end{align}
where $\phi_k[\mu] \in \Cc^{p}(\Xx)$ is the Schr\"odinger potential from $\mu$ to $\nu_k$.
The map $\mu\mapsto V[\mu]$ is Lipschitz continuous in the sense that there exists $L>0$ such that
\begin{align*}
\Vert V[\mu] - V[\mu']\Vert_{\Cc^{p-1}} \leq L\, W_2(\mu,\mu'), \qquad \forall \mu,\mu'\in \Pp(\Xx).
\end{align*}
We moreover have that for $p'\leq p$, there exists $C_{p'}>0$ independent of $\lambda>0$, $(w_k)_k$ and $(\nu_k)_k$ such that  $\Vert V[\mu] \Vert_{\tilde \Cc^{p'}}\leq C_{p'} \lambda^{\min\{0,1-p'\}}$.
\end{proposition}
Conveniently, the objective is also strongly convex for the total variation norm $\Vert \cdot \Vert_{\TV}$ \rev{defined, for a signed measure $\sigma$ on $\Xx$ with absolute variation $\vert \sigma\vert$ as $\Vert \sigma\Vert_\TV = \vert \sigma\vert(\Xx)$}.

\begin{proposition}[Strong convexity of $F_{\tau,\lambda}$]\label{prop:strong-convexity}
For $\lambda\geq 0$, the objective $F_{\lambda,\tau}=G_\lambda+\tau H$ is $\tau$-strongly convex on $\Pp(\Xx)$ for the total variation norm, \rev{in the sense that for all $\theta \in [0,1]$ and $\mu_0,\mu_1\in \Pp(\Xx)$ it holds 
\begin{align*}
F_{\lambda,\tau}(\theta\mu_0+(1-\theta)\mu_1) \leq \theta F_{\lambda,\tau}(\mu_0) + (1-\theta)F_{\lambda,\tau}(\mu_1)-\frac{\tau}{2}\theta(1-\theta)\Vert \mu_1-\mu_0\Vert_{\TV}^2.
\end{align*}}
\end{proposition}
\begin{proof}
It is well known that $\mu\mapsto H(\mu)$ is $1$-strongly convex over $\Pp(\Xx)$ for the total variation norm (are proof using Pinsker's inequality is provided in Lem.~\ref{lem:H-strongly-convex} for completeness). 
Since $F_{\lambda,\tau}=G_\lambda+\tau H$ and $G_\lambda$ is convex, the result follows. 
\end{proof}
Strong convexity in $\ell^2$-norm on a discrete space $\Xx$ (which follows from Prop.~\ref{prop:strong-convexity} since then the total variation norm is the $\ell^1$ norm, and $\Vert \cdot \Vert_{\ell^2}\leq \Vert \cdot \Vert_{\ell^1}$) was already shown for the $(\lambda,\lambda)$-barycenter functional in~\cite[Thm.~3.4]{bigot2019data} with a technical proof tailored to that specific case. The equivalent formulation as a doubly-regularized problem makes this property immediate.

As a consequence of all these regularity results, we now show that the $(\lambda,\tau)$-barycenters can be expressed as a smooth Gibbs density, that solves a fixed point problem.
\begin{assbox}
\begin{theorem}[Existence and regularity]\label{thm:characterization}
For any $\lambda,\tau > 0$, $F_{\lambda,\tau}$ admits a unique minimizer $\mu^*_{\lambda,\tau} \in \Pp(\Xx)$. It is an absolutely continuous measure with density
\begin{align}\label{eq:characterization}
\frac{\d \mu^*_{\lambda,\tau}}{\d x} \propto  e^{-V[\mu^*_{\lambda,\tau}]/\tau} 
\end{align}
where $V[\mu^*_{\lambda,\tau}]=\sum_{k=1}^K w_k \phi_k[\mu^*_{\lambda,\tau}] \in \Cc^{p}(\Xx)$ satisfies the regularity estimates of Prop.~\ref{prop:Reg_G}. \rev{Moreover, if $\mu\in \Pp(\Xx)$ is absolutely continuous and satisfies~\eqref{eq:characterization}, then $\mu=\mu^*_{\lambda,\tau}$.}
\end{theorem}
\end{assbox}
\begin{proof}
The functional $G_\lambda$ is weakly continuous and $H$ is weakly lower-semicontinuous~\cite[Sec.~7.1.2]{santambrogio2015optimal} so $F_{\lambda,\tau}$ is weakly lower-semicontinuous. It is not identically $+\infty$ since it takes a finite value for the normalized Lebesgue measure on $\Xx$. Since $\Pp(\Xx)$ is weakly compact, we deduce from the direct method of the calculus of variations that there exists at least one minimizer $\mu^*_{\lambda,\tau}$. Moreover $F_{\lambda,\tau}$ is strictly convex (Prop.~\ref{prop:strong-convexity}), so the minimizer is unique and since $H(\mu^*_\lambda)<\infty$ this measure is absolutely continuous. \rev{Finally the implicit expression of the minimizer (Eq.~\eqref{eq:characterization}) and the fact that it is a sufficient optimality condition is proved in Prop.~\ref{prop:general-min-entropy-pb}.} 
\end{proof}
Interestingly, the minimizer of $F_{\lambda,\tau}$ for $\tau>0$ has always full support in $\Xx$ even when none of the marginals $\nu_1,\dots,\nu_K$ have. This property \rev{-- which may or may not be desirable depending on contexts --} is not satisfied when $\tau=0$ or for the Sinkhorn Divergence barycenter (see Section \ref{sec:gaussians} for the case of point mass marginals). 

\paragraph{Relation to prior works} To the best of our knowledge, little is known on the regularity of $(0,0)$-barycenters beyond absolute continuity under the condition that at least one $\nu_k$ is absolutely continuous~\cite{agueh2011barycenters}, \cite[Thm.~5.1]{kim2017wasserstein}. The only other regularity result we are aware of concerns $(0,\tau)$-barycenters and the quadratic cost:~\cite{carlier2021entropic} show a Fisher Information bound on $\mu^*_{0,\tau}$ (their Lem.~4.1) that imply Lipschitz regularity \rev{of the log-density of the barycenter} in the compact case : this is consistent with the $\lambda\to 0$ limit of Thm.~\ref{thm:characterization}. They additionally show (their Prop.~5.2) that $\mu^*_{0,\tau}$  gains two degrees of regularity compared to the densities $\nu_k$. With $\lambda,\tau>0$, we see from Thm.~\ref{thm:characterization} that the $(\lambda,\tau)$-barycenters are as regular as the cost function, irrespective of the regularity of the marginals.

\subsection{Dual formulations}\label{sec:dual-formulations}
There are several ways to derive a dual formulation for~\eqref{eq:optim-objective}. Let us detail one of them which stands out as an elegant composition of two soft-max (log-sum-exp) functions -- which is the dual consequence of the double regularization. This dual formulation can be used to compute $(\lambda,\tau)$-barycenters in practice, as done in Section~\ref{sec:numerics}. \rev{In this statement, we say that a functional on a normed space $E:X\to \mathbb{R}$ is $\alpha$-smooth if it is twice Fréchet differentiable and for all $x,y\in X$, $\vert D^2 E[x](y,y)\vert \leq \alpha \Vert y\Vert^2$.}
\newcommand{\bphi}{{\boldsymbol{\phi}}}
\newcommand{\bpsi}{{\boldsymbol{\psi}}}
\begin{proposition}[Dual formulation]\label{prop:dual}
For $\lambda,\tau>0$, one has
\begin{align}\label{eq:primal-dual}
\min_{\mu \in \Pp(\Xx)} F_{\lambda,\tau}(\mu)=\max_{\bpsi \in \Cc(\Xx)^K} E(\bpsi)
\end{align}
where $E(\bpsi) $ is defined as
\begin{align}\label{eq:dual}
\sum_{k=1}^K w_k \int_\Xx \psi_{k}\d\nu_k -\tau \log \left\{
\int_\Xx \exp 
\left[
\frac{\lambda}{\tau} \sum_{k=1}^Kw_k \log\int_\Xx \exp\Big( \frac{\psi_k(y)-c(x,y)}{\lambda}\Big)\d\nu_k(y) 
\right]\d x
\right\}.
\end{align}
The function $E$ is concave, $1$-Lipschitz continuous and $\min\{\lambda,\tau\}^{-1}$-smooth for the seminorm $\bpsi \mapsto \sum w_k \Vert \psi_k\Vert_\mathrm{osc}$ where $\Vert \psi\Vert_\mathrm{osc} = \sup_y \psi(y)-\inf_y \psi(y)$. It admits a maximizer with sup-norm smaller than $\Vert c\Vert_{\mathrm{osc}}$ (which is unique up to shifting each $\psi_k$ by constants). 
Moreover, the barycenter $\mu^*_{\lambda,\tau}$ is the Gibbs distribution associated to the solution of the dual problem (see~\eqref{eq:mu-from-psi}). 
\end{proposition}
\begin{proof}
Let us start from the objective of the dual formulation~\eqref{eq:dual-EOT} of $T_\lambda(\mu,\nu)$. Maximizing it over $\phi$ gives $\phi=\phi_{\psi,\nu}$ $\nu$-a.e.\ with 
\begin{align}\label{eq:soft-c-transform}
\phi_{\psi,\nu}(x) = -\lambda \log \Big( \int e^{(\psi(y)-c(x,y))/\lambda}\d\nu(y)\Big).
\end{align}
The so-called ``semi-dual'' formulation of EOT follows
\begin{align}\label{eq:semi-dual-EOT}
T_\lambda(\mu,\nu) = \max_{\psi\in \Cc(\Xx)} \int \psi\d\nu +\int \phi_{\psi,\nu}\d\mu.
\end{align}
We can thus rewrite the objective of the barycenter in Lagrangian form as
\begin{align*}
\min_{\mu\in \Pp(\Xx)} \max_{\bpsi \in \Cc(\Xx)^K} \Big\{L(\mu,\bpsi) \coloneqq \sum_{k=1}^K w_k \int  \psi_k\d\nu_k + \int  \Big( \sum_{k=1}^K w_k \phi_{\psi_k,\nu_k}\Big)\d\mu +\tau H(\mu)\Big\}.
\end{align*}
Observe that $\psi\mapsto \phi_{\psi,\nu}$ is a $1$-Lipschitz continuous function for the supremum norm and therefore, $\bpsi \mapsto L(\mu,\bpsi)$ is also $1$-Lipschitz continuous for the norm $\Vert \bpsi\Vert_\infty \coloneqq \max_k \Vert \psi_k\Vert_\infty$.
\rev{In order to exchange min and max, we will verify the assumptions of Sion's minimax theorem~\cite[Cor.~3.3]{sion1958general}:
\begin{theorem}[Sion's minimax theorem]\label{thm:Sion}
Let $X$ be a compact convex subset of a linear vector space and $Y$ a convex subset of a linear vector space. If $L:X\times Y\to \RR$ satisfies
\begin{enumerate}[(i)]
\item For all $x\in X$, $L(x,\cdot)$ is upper semicontinuous and concave on $Y$,
\item For all $y \in Y$, $L(\cdot,y)$ is lower semicontinuous and convex on $X$
\end{enumerate}
then $\min_{x\in X}\sup_{y\in Y} L(x,y) = \sup_{y\in Y} \min_{x\in X} L(x,y)$.
\end{theorem}
In our context, $X$ is $\Pp(\Xx)$ endowed with the weak topology (seen as a subset of the vector space of finite signed Borel measures) and $Y$ is $\Cc(\Xx)^K$ endowed with the sup-norm. The function $L$ is convex and lower semicontinuous in $\mu$ (for the weak topology) and it is concave and continuous in $\bpsi$ (for the sup-norm). Since in addition $\Pp(\Xx)$ is weakly compact, we can apply Sion's minimax theorem to exchange the order of min and max.

 Minimizing $L$ in $\mu \in \Pp(\Xx)$ gives, by a direct application of Prop.~\ref{prop:general-min-entropy-pb} in the special case where the convex functional is linear, a unique minimizer denoted by $\mu_\bpsi$ which is absolutely continuous with density}
\begin{align}\label{eq:mu-from-psi}
\frac{\d \mu_\bpsi}{\d x} = e^{(\chi_\bpsi-V_\bpsi)/\tau}&&\text{with}&&V_\bpsi\coloneqq \sum_{k=1}^K w_k \phi_{\psi_k,\nu_k}&&\text{and}&& \quad\chi_\bpsi\coloneqq-\tau \log \int e^{-V_\bpsi/\tau}\d x
\end{align}
and the objective becomes
$
\sup_{\bpsi \in \Cc(\Xx)^K} \sum_{k=1}^K w_k \int  \psi_k\d\nu_k + \chi_\bpsi
$
which is exactly Eq.~\eqref{eq:dual}.


\rev{Let us prove the remaining properties by applying Lem.~\ref{lem:LSE} below, which gathers classical properties of the log-sum-exp operator (see e.g.~\cite[Ex.~3.14]{boyd2004convex}).} The function $E$ in~\eqref{eq:dual} is concave, as minus the composition of: an affine function, by a LSE, by a sum with nonnegative weights, by a LSE. It is also $1$-Lipschitz continuous, as a composition of $1$-Lipschitz continuous functions and its differential is given for $\delta \psi_k\in \Cc(\Xx)$ by
\begin{align*}
DE[\bpsi](\delta \psi_k) = w_k\int \delta \psi_k\d\nu_k - w_k\int \delta \psi_k(y)e^{(\phi_{\psi_k,\nu_k}(x)+\psi_k(y)-c(x,y))/\lambda +(\chi_\bpsi(x)-V_\bpsi(x))/\tau}\d x \d\nu_k(y)
\end{align*}
Note that the second integral is the integral of $\delta \psi_k$ against a probability measure $\gamma_{\psi_k}(\d x,\d y) \in \Pi(\mu_\bpsi,\nu_k)$ which can be disintegrated as $\gamma_{\psi_k}(\d y|x)\mu_\bpsi(\d x)$. Differentiating once more, we obtain
\begin{align*}
D^2E[\bpsi](\delta \psi,\delta \psi) =-\frac{1}{\tau} \mathbf{Var}_{x\sim \mu_\bpsi}\Big[\sum_{k=1}^K w_k \E_{\gamma_{\psi_k}(\cdot |x)} [\delta \psi_k]\Big] -\frac{1}{\lambda}\E_{x\sim \mu_\bpsi}\Big[\sum_{k=1}^K w_k \mathbf{Var}_{\gamma_{\psi_k}(\cdot|x)}(\delta \psi_k)\Big].
\end{align*}
It follows by the law of total variance that 
$$
-\frac{1}{ \max\{\tau,\lambda\}} \mathbf{Var}\Big[  \sum_{k=1}^K w_k\delta \psi_k(y_k) \Big] \geq D^2E[\bpsi](\delta \psi,\delta \psi) \geq -\frac{1}{ \min\{\tau,\lambda\}} \mathbf{Var}\Big[  \sum_{k=1}^K w_k\delta \psi_k(y_k) \Big] 
$$
where the variance is under $(x,y_1,\dots,y_k)$ such that each couple $(x,y_k)$ is distributed according to $ \gamma_{\psi_k}$. This shows that $E$ is $\frac{1}{ \min\{\tau,\lambda\}}$-smooth for the seminorm $\bpsi\mapsto \sum w_k \Vert \psi_k\Vert_\mathrm{osc}$, and also shows the uniqueness of the dual solution up to constant shifts. Finally, it can be checked that the Schr\"odinger potentials associated to the minimizer~\eqref{eq:characterization} satisfy the optimality conditions, so they are maximizers of $E$ and have a sup-norm bounded by $\Vert c\Vert_{\mathrm{osc}}$.
\end{proof}
\rev{
\begin{lemma}[Properties of log-sum-exp]\label{lem:LSE}
For $\nu\in \Pp(\Xx)$ and $\alpha>0$ consider the log-sum-exp operator $\LSE_\nu : 
\Cc(\Xx\times \Xx) \to \Cc(\Xx)$ defined for $f\in \Cc(\Xx\times \Xx)$ and $x\in \Xx$ by
\begin{align}
\LSE_\nu(f)(x) = \alpha \log \int e^{f(x,y)/\alpha}\d\nu(y).
\end{align}
Then the following properties are satisfied:
\begin{enumerate}[(i)]
\item if $f(\cdot,y)$ is convex for every $y\in \Xx$ then $x\mapsto \mathrm{LSE}_\nu(f)(x)$ is a convex function;
\item $\mathrm{LSE}_\nu$ is $1$-Lipschitz continuous and  $1/\alpha$-smooth for the sup-norm.
\end{enumerate}
\end{lemma}
\begin{proof}
Let $f\in \Cc(\Xx\times \Xx)$ be such that $f(\cdot,y)$ is convex for any $y \in \Xx$. Then for any $\theta \in ]0,1[$ and $x,x'\in \Xx$ it holds
\begin{align*}
\LSE_\nu(f)(\theta x+(1-\theta)x') &= \alpha \log \int e^{f(\theta x+(1-\theta)x',y)/\alpha}\d\nu(y)\\
&\leq  \alpha \log \int e^{(\theta f(x,y)+(1-\theta)f(x',y))/\alpha}\d\nu(y)\\
&\leq  \alpha \log \Big\{\Big(\int e^{f(x,y)/\alpha}\d\nu(y)\Big)^{\theta}\Big(\int e^{f(x',y)/\alpha}\d\nu(y)\Big)^{1-\theta}\Big\}\\
&=\theta \LSE_\nu(f)(x) + (1-\theta) \LSE_\nu(f)(x')
\end{align*}
where we have used Hölder's inequality (with exponents $(p,q)=(1/\theta,1/(1-\theta))$) in the third line. This proves the claim (i).

The log-sum-exp function is a composition of the map $f\mapsto e^{f/\alpha}$ which is\footnote{The notation $\Cc_{>0}(\Xx)$ represents here the subset of $\Cc(\Xx)$ of the continuous and positive functions over $\Xx$.} $\Cc^\infty(\Cc(\Xx\times \Xx);\Cc_{>0}(\Xx\times \Xx))$, the map $g\mapsto \int g(\cdot,y)\d\nu(y))$ which is a continuous (positive) linear map from $\Cc(\Xx\times \Xx)$ to $\Cc(\Xx)$ and the map $h\mapsto \alpha \log(h)$ which is $\Cc^\infty(\Cc_{>0}(\Xx );\Cc(\Xx))$. Therefore, as a composition of these three maps, $\mathrm{LSE}_\nu$ is $\Cc^\infty(\Cc(\Xx\times \Xx);\Cc(\Xx))$.
Then Claim (ii) can be seen from the expression of the first and second order Fréchet differentials which, for $\delta f \in \Cc(\Xx\times \Xx)$ can be computed as
\begin{align*}
D\mathrm{LSE}_\nu[f](\delta f) = x\mapsto \E_{Y_x}[ \delta f(x,Y_x)] && D^2\mathrm{LSE}_\nu[f](\delta f,\delta f) = x \mapsto \frac{1}{\alpha}\mathbf{Var}_{Y_x}[\delta f(x,Y_x)]
\end{align*}
where $\E_{Y_x}$ (resp.~$ \mathbf{Var}_{Y_x}$) denote the expectation (resp.~centered variance) for $Y_x$ distributed according to the probability measure proportional to $e^{f(x,y)/\alpha}\nu(\d y)$. Clearly $\Vert D\mathrm{LSE}_\nu[f](\delta f)\Vert_\infty\leq \Vert \delta f\Vert_\infty$, so $\LSE_\nu$ is $1$-Lipschitz continuous. Moreover, $\Vert D^2\mathrm{LSE}_\nu[f](\delta f,\delta f)\Vert_\infty\leq \alpha^{-1} \Vert \delta f\Vert_\infty^2$ so $\LSE_\nu$ is $1/\alpha$-smooth as we recognize a characterization of smoothness for twice Fréchet differentiable functions.
\end{proof}
}

\rev{\paragraph{Dual formulations when $\tau\geq \lambda$.} For the sake of completeness, let us mention other useful dual formulations that become available when $\tau\geq \lambda$. Using Eq.~\eqref{eq:change-ref-measure} (below) to change the reference measure from $\mu\otimes \nu_k$ to $(\d x)\otimes \nu_k$, the problem defining $\mu^*_{\lambda,\tau}$ can then be rewritten (up to constants that we ignore)
$$
\min_{\mu \in \Pp(\Xx)} \min_{\gamma_k\in \Pi(\mu,\nu_k)} \int c(x,y)\d\gamma_k(x,y)+\lambda \KL(\gamma_k|\d x \otimes \nu_k) + (\tau-\lambda)H(\mu)
$$
Using the dual formulation of EOT, this problem becomes
$$
\min_{\mu \in \Pp(\Xx)} \max_{\bphi,\bpsi} \sum_k w_k \int \varphi_k\d\mu +\int \psi_k \d\nu_k +\lambda \Big( 1-\int e^{(\phi_k(x)+\psi_k(y)-c(x,y))/\lambda}\d x \d\nu_k(y) \Big) + (\tau-\lambda) H(\mu)
$$

\begin{itemize}
\item  When $\tau=\lambda$, exchanging min/max and minimizing over $\mu\in \Pp(\Xx)$ and one gets the dual problem~\cite[Prop.~9.1]{peyre2019computational}
\begin{align*}
\max_{\bphi,\bpsi} &\quad \sum_{k=1}^K w_k  \int \psi_k\d\nu_k + \lambda \sum_{k}^K w_k \Big( 1-\int e^{(\phi_k(x)+\psi_k(y)-c(x,y))/\lambda}\d x \d \nu_k(y)  \Big)\\ \text{subject to} &\quad \sum_{k=1}^K w_k \phi_k=0
\end{align*}
with, at optimality, $\mu^*_{\lambda,\lambda}(\d x) = \int e^{(\phi_k(x)+\psi_k(y)-c(x,y))/\lambda} \d \nu_k(y)\d x$ for any $k\in \{1,\dots,K\}$. Alternate maximization on the blocks $\bphi$ and $\bpsi$ leads to a convenient Sinkhorn-like algorithm when the Lebesgue measure is discretized (see~\cite{kroshnin2019complexity} for a complexity analysis).
\item When $\tau> \lambda$, again exchanging min/max and minimizing over $\mu\in \Pp(\Xx)$ leads to another dual formulation which was proposed in~\cite[Prop.~3.2]{ballu2020stochastic}:
\begin{multline*}
\max_{\bphi,\bpsi} \quad \sum_{k=1}^K w_k  \int \psi_k\d\nu_k + \lambda \sum_{k}^K w_k \Big( 1-\int e^{(\phi_k(x)+\psi_k(y)-c(x,y))/\lambda}\d x \d \nu_k(y)  \Big)\\
-(\tau-\lambda)\log \int e^{-\frac{\sum_{k=1}^K w_k \phi_k(x)}{\tau-\lambda}}\d x
\end{multline*}
with, at optimality, $\mu^*_{\lambda,\tau}(\d x) \propto e^{-\frac{\sum_{k=1}^K w_k \phi_k(x)}{\tau-\lambda}}\d x$. This formulation is exploited in~\cite{ballu2020stochastic} to derive an efficient stochastic optimization scheme.
\end{itemize}
} 
\subsection{Extensions}
Let us conclude this section with a discussion of our setting and potential extensions. 
\begin{itemize}
\item (General ambient space) The definition of $(\lambda,\tau)$-barycenters would make sense in the more general context where $\Xx$ is a Polish space with an outer regularization $\KL(\cdot|\mu_{\mathrm{ref}})$ with a reference measure $\mu_{\mathrm{ref}}\in \Pp(\Xx)$ replacing the Lebesgue measure. In particular, the compactness assumption is not necessary, provided that the cost satisfies certain integrability conditions (see~\cite{nutz2021introduction} for a review of EOT under weak assumptions). In this paper, we focus on the compact case on $\RR^d$ for simplicity and because, to date, Prop.~\ref{prop:propertiesEOT}-(i) and Prop.~\ref{prop:expansion}-(i) which we use below are only known in this setting.
\item (Infinite number of marginals) The problem of Wasserstein barycenter is often formulated~\cite{agueh2017vers} in the more general form where the EOT barycenter functional is an expectation under some distribution $P\in \Pp(\Pp(\Xx))$ instead of a finite sum, i.e.~
\begin{align}\label{eq:p-of-p}
G_\lambda(\mu) = \int_{\Pp(\Xx)} T_\lambda(\mu,\nu)\d P(\nu).
\end{align}
The $(\lambda,\tau)$-barycenters could also be studied in this setting, where interesting questions arise related to estimation rates and stability.
\end{itemize}

\section{Approximating the Wasserstein barycenter}\label{sec:approximation}
In this section, we study the approximation error, that is the difference between $(\lambda,\tau)$-barycenters and the $(0,0)$-barycenter. We also prove approximation error bounds for the Sinkhorn divergence barycenter as well. Our goal is to show that the double regularization is not just a convenient trick to obtain nice properties for $\mu^*_{\lambda,\tau}$, since it also helps approximating $\mu^*_{0,0}$ better when $\lambda\to 0$, in particular when $\tau\sim \nicefrac{\lambda}{2}$. We mention however that we do not specially advocate choosing a small $\lambda$ in practice, as the other desirable properties of $(\lambda,\nicefrac{\lambda}2)$-barycenters degrade very quickly as $\lambda$ decreases (see~\cite{chizat2020faster} for an analysis of the trade-offs in choosing $\lambda$ in a similar context).

\subsection{Reduced approximation error for smooth densities}
In this section we discuss the case of the quadratic cost and smooth marginals. At the heart of our approximation result is the following known comparison between $T_\lambda$ and $T_0$.

\begin{proposition}\label{prop:expansion}
Assume that $\mu$ and $\nu$ have bounded densities on $\Xx$ and let $c(x,y)=\frac12 \Vert y-x\Vert_2^2$. Then
\begin{align}\label{eq:entropic-expansion-bound}
T_0(\mu,\nu)\leq T_\lambda(\mu,\nu) + \frac{d\lambda}{2}  \log(2\pi\lambda) + \frac{\lambda}{2} (H(\mu)+H(\nu)) \leq   T_0(\mu,\nu) + \frac{\lambda^2}{8}I(\mu,\nu) 
\end{align}
where $I(\mu,\nu)$ is the integrated Fisher information of the Wasserstein geodesic $(\rho_t\d x)_{t\in [0,1]}$ that connects $\mu$ to $\nu$, i.e. $I(\mu,\nu)\coloneqq \int_0^1\int_\Xx \Vert \nabla \log \rho_t(x)\Vert^2\rho_t(x)\d x\d t\geq 0$. Moreover, if $I(\mu,\nu)<\infty$ then
\begin{align}\label{eq:entropic-expansion}
 T_\lambda(\mu,\nu) + \frac{d\lambda}{2}  \log(2\pi\lambda) + \frac{\lambda}{2} (H(\mu)+H(\nu)) =   T_0(\mu,\nu) + \frac{\lambda^2}{8}I(\mu,\nu)  +o(\lambda^2).
\end{align}
\end{proposition}
The logarithmic derivative $\nabla \log \rho_t$ appearing in the statement is the density of the distributional gradient $\nabla \rho_t$ with respect to $\rho_t$ when it exists, and $I_0(\mu,\nu)=+\infty$ if this quantity is not defined for a.e. $t\in [0,1]$. As shown in~\cite[Thm.~1]{chizat2020faster}, the first claim (Eq.~\eqref{eq:entropic-expansion-bound}) is a direct consequence of a dynamical formulation of $T_\lambda$~\cite{chen2016relation}. The second claim (Eq.~\eqref{eq:entropic-expansion}) was proved in~\cite[Lem.~1]{chizat2020faster} and then in~\cite{conforti2021formula} (who first formulated the ansatz) in a more general setting. It is a more precise version of previous first-order expansions~\cite{duong2013wasserstein, erbar2015large, pal2019difference}. Let us also mention that~\cite[Prop.~1]{chizat2020faster} gives a priori bounds on $I(\mu,\nu)$ in terms of the derivatives up to order $3$ of Kantorovich potentials. 

\rev{We deduce from Eq.~\eqref{eq:entropic-expansion} that, if $\mu,\nu_1,\dots,\nu_K$ have bounded densities, it holds
$$
F_{\lambda,\tau}(\mu) = G_0(\mu) +\big(\tau-\frac{\lambda}{2}\big)H(\mu) +\frac{\lambda^2}{8} \sum_{k=1}^K w_k I(\mu,\nu_k) +o(\lambda^2) + C
$$
for some $C\in \RR$ that does not depend on $\mu$, and thus does not affect the minimizers. This shows that the choice $\tau=\lambda/2$ stands out, as this cancels exactly the leading order error term between $F_{\lambda,\tau}$ and $F_{0,0}=G_0$. Of course, the same holds for any choice of $\lambda,\tau\to 0$ that is such that $\tau \sim \nicefrac{\lambda}{2}$ although we focus on $\tau=\lambda/2$ in the following for simplicity.} An approximation bound in terms of suboptimality gap for the Wasserstein barycenter functional easily follows.
\begin{assbox}
\begin{theorem}[Approximation bound]\label{thm:approx}
Assume that $\nu_1,\dots,\nu_K$ have bounded densities and let $c(x,y)=\frac12\Vert y-x\Vert_2^2$. Then for $\lambda>0$ the $(\lambda,\nicefrac{\lambda}2)$-barycenter satisfies
\begin{align}
G_0(\mu^*_{\lambda,\lambda/2})-G_0(\mu^*_{0,0})\leq \frac{\lambda^2 }{8} \sum_{k=1}^K w_k I(\mu^*_{0,0},\nu_k)
\end{align}
where $I$ is defined in Prop.~\ref{prop:expansion}. For the Sinkhorn divergence barycenter \rev{$\mu^{*}_{\lambda,\mathrm{div}}$}, it holds:
$$
G_0(\mu^{*}_{\lambda,\mathrm{div}})-G_0(\mu^*_{0,0})\leq \frac{\lambda^2}{8}\sum_{k=1}^K w_k \Big(I(\mu^*_{0,0},\nu_k) +\frac12 I(\mu^*_{\lambda,\mathrm{div}})+\frac12 I(\nu_k)\Big)
$$
where $I(\mu)\coloneqq I(\mu,\mu)$ denotes the Fisher information of $\mu$.
\end{theorem}
\end{assbox}
\begin{proof}
First notice that Prop.~\ref{prop:expansion} indeed applies for any couple of the form $(\mu^*_{\lambda,\lambda/2},\nu_k)$: the bounded density assumption holds by Thm.~\ref{thm:characterization} for $\lambda,\tau>0$ and by~\cite[Thm.~5.1]{agueh2011barycenters} for $\lambda=0$ since the Wasserstein barycenter has a bounded density.
Let us call $\tilde T_\lambda$ the quantity that is sandwiched in Eq.~\eqref{eq:entropic-expansion-bound} and let us define $\tilde F_\lambda(\mu) \coloneqq \sum_{k=1}^K w_k\tilde T_\lambda(\mu,\nu_k)$ which differs from $F_{\lambda,\lambda/2}$ only by a constant. Since $\mu^*_{\lambda,\lambda/2}$ is the minimizer of $F_{\lambda,\lambda/2}$, it is also the minimizer of $\tilde F_\lambda$, so for any $\lambda\geq 0$, $\tilde F_\lambda(\mu^*_{\lambda,\lambda/2})\leq \tilde F_\lambda(\mu^*_{0,0})$.
By Prop.~\ref{prop:expansion}, it holds for any $\mu\in \Pp_2(\Xx)$ \rev{with a bounded density} that
 $$ 0\leq \tilde T_\lambda(\mu,\nu_k) - T_0(\mu,\nu_k) \leq \frac{\lambda^2}{8}I(\mu,\nu_k).$$ 
Taking the weighted sum over $k\in\{1,\dots, K\}$, we get $0\leq  \tilde F_\lambda(\mu)-G_0(\mu) \leq \frac{\lambda^2}{8}\sum_{k=1}^K w_k I(\mu,\nu_k)$. It follows
\begin{multline*}
G_0(\mu^*_{\lambda,\lambda/2}) - G_0(\mu^*_{0,0})\\
\leq [G_0(\mu^*_{\lambda,\lambda/2}) - \tilde F_\lambda(\mu^*_{\lambda,\lambda/2})]
 + [ \tilde F_\lambda(\mu^*_{\lambda,\lambda/2}) -  \tilde F_\lambda(\mu^*_{0,0})]
 +[\tilde F_\lambda(\mu^*_{0,0}) - G_0(\mu^*_{0,0})]\\
\leq 0 + 0 + \frac{\lambda^2}{8}\sum_{k=1}^K w_k I(\mu^*_{0,0},\nu_k).
\end{multline*}

As for the Sinkhorn divergence barycenter, it is the minimizer of $G^\mathrm{div}_\lambda  \coloneqq \sum_{k=1}^K w_kS_\lambda(\cdot,\nu_k)$ where
$
S_{\lambda}(\mu,\nu) \coloneqq T_\lambda(\mu,\nu) - \frac12 T_\lambda(\mu,\mu)-\frac12 T_\lambda(\nu,\nu)
$
is the Sinkhorn divergence. After manipulating inequalities~\eqref{eq:entropic-expansion-bound}, we obtain for any $\mu,\nu$ \rev{with bounded densities} that
$$
T_0(\mu,\nu) -\frac{\lambda^2}{16}(I(\mu)+I(\nu)) \leq S_{\lambda}(\mu,\nu) \leq T_0(\mu,\nu) +\frac{\lambda^2}{8}I(\mu,\nu).
$$
As before, we sum these inequalities over $k\in \{1,\dots,K\}$ and get
\begin{align*}
G_0(\mu^*_{\lambda,\mathrm{div}}) - G_0(\mu^*_{0,0})&\leq [G_0(\mu^*_{\lambda,\mathrm{div}}) - G^\mathrm{div}_{\lambda}(\mu^*_{\lambda,\mathrm{div}})]
 + [ G^\mathrm{div}_{\lambda}(\mu^*_{\lambda,\mathrm{div}}) -  G^\mathrm{div}_{\lambda}(\mu^*_{0,0})]\\
 &\qquad\qquad\qquad\qquad\qquad\qquad\;\;
 + [G^\mathrm{div}_{\lambda}(\mu^*_{0,0}) - G_0(\mu^*_{0,0})]\\
&\leq \frac{\lambda^2}{16}I(\mu^*_{\lambda,\mathrm{div}}) +\frac{\lambda^2}{16}\sum_{k=1}^K w_k I(\nu_k) + 0 + \frac{\lambda^2}{8}\sum_{k=1}^K w_k I(\mu^*_{0,0},\nu_k).\qedhere
\end{align*}
\end{proof}

Regarding the generality of this result, we can also make the following comments:
\begin{itemize}
\item the choice $\tau=\lambda/2$ leads to approximation benefits  for all costs of the form $(\alpha/2)\Vert y-x\Vert_2^2$ for $\alpha>0$. Indeed, insisting on the dependency in $c$ in the notation, we can always bring ourselves back to the case $\alpha=1$ using that $
F_{\alpha c,\alpha \lambda,\alpha \tau} = \alpha F_{c,\lambda,\tau}.
$
\item For costs with a non-constant Hessian, there appears to be no explicit way to debias the problem since in those cases, as shown in~\cite{pal2019difference}, the first order term in the expansion~\eqref{eq:entropic-expansion} depends on the optimal transport map itself, which is unknown a priori.
\item Note that these bounds involve quantities related to the regularity of the unregularized barycenter for which no a priori bound exist unfortunately.
\end{itemize}

\paragraph{Approximation in $W_2$ distance} While Thm.~\ref{thm:approx} gives approximation bounds in terms of suboptimality gap, one could wish to state approximation bounds in terms of a notion of distance. There is an active line of work on the stability of Wasserstein barycenters (see~\cite{carlier2024displacement} for recent advances). Let us give a result suggesting that in the most favorable settings our bound in Thm.~\ref{thm:approx} gives a direct upper-bound on the squared $L_2$-Wasserstein distance from the barycenter $\mu^*_{0,0}$. The following proposition can be found in~\cite[Thm.~6]{chewi2020gradient}, following earlier results in~\cite{ahidar2020convergence}.

\begin{proposition}
Assume that for each $k\in \{1,\dots,K\}$, there exists a Brenier potential $\tilde \phi_k = \frac12\Vert x\Vert^2 - \phi_k$ from $\mu_{0,0}^*$ to $\nu_k$ that is $\alpha_k>0$ strongly convex over $\Xx$. Then the following \emph{variance inequality} holds
$$
\frac12 W_2^2(\mu,\mu^*_{0,0}) \leq \frac{G_0(\mu) - G_0(\mu^*_{0,0})}{\sum_{k=1}^K w_k \alpha_k}, \quad \forall \mu \in \Pp(\Xx).
$$
\end{proposition}
As shown in~\cite{chewi2020gradient}, this proposition holds in the case of the barycenter between Gaussian measures. In this case if for each $\nu_k=\Nn(a_k,A_k)$ with $a_k\in \RR^d$ and $A_k$ a positive definite matrix, it holds $\Vert A_k\Vert_{\mathrm{op}}\leq 1$ and $\det A_k\geq \zeta$ then the upper-bound is $\zeta^{-1}(G_0(\mu)-G_0(\mu^*_{0,0}))$. Unfortunately, the Gaussian case is not covered by Thm.~\ref{thm:approx} due to the compactness assumption, but inspecting their proof, it can be seen that the result in fact applies to barycenters of \emph{elliptically-contoured distributions} in the same family, including smooth compact cases with finite Fisher information, hence covered by Thm.~\ref{thm:approx}. See~\cite[Prop.~14]{chizat2020faster} for an explicit example of such a class of distributions.

\subsection{Closed-form for isotropic Gaussian distributions}\label{sec:gaussians}
In this section, we leverage closed-form expressions of entropic OT for Gaussian measures~\cite{chen2016relation, mallasto2022entropy, del2020statistical,janati2020entropic} to get a finer understanding of the role of $\lambda$ and $\tau$ in the approximation error. In order to get closed-form solutions, we focus on the simplest case of the barycenter between a family of isotropic Gaussian distributions with equal variance. Our computations, detailed in Appendix~\ref{app:gaussians}, follow those of~\cite{janati2020debiased}  and extend them by introducing a general parameter $\tau$.

\begin{proposition}\label{prop:gaussians}
For $k\in \{1,\dots,K\}$, let $\nu_k = \Nn(x_k,aI_d)$, for $x_k\in \RR^d$ and $a>0$. Then the $(\lambda,\tau)$-barycenter $\mu^*_{\lambda,\tau}$ is the Gaussian $\Nn(\bar x,bI_d)$ where $\bar x=\sum_{k=1}^K w_kx_k$ and the variance is
\begin{align}\label{eq:gaussian-variance}
b = \frac{\big(a+\sqrt{(a-\lambda)^2+4a\tau}\big)^2-\lambda^2}{4a}.
\end{align}
In particular, one has $b=a$ (i.e. exact debiasing) with the choice
\begin{align}\label{eq:nonasymp-debias}
\tau= \tau^*(\lambda) &= \frac{\lambda}{2} + a\Big(1-\Big(1+\frac{\lambda^2}{4a^2}\Big)^{1/2}\Big) =\frac{\lambda}{2} -\frac{\lambda^2}{8a} +O(\lambda^4/a^3).
\end{align}
\end{proposition}
We report various special or limit cases of~\eqref{eq:gaussian-variance} in Table~\ref{tab:gaussian}, as well as the case of Sinkhorn divergence barycenters. The latter does not follow from~\eqref{eq:gaussian-variance} and is taken from~\cite{janati2020debiased}, which also covered the special cases $\tau=0$ and $\tau=\lambda$. 

\begin{figure}
\centering
\includegraphics[scale=0.6]{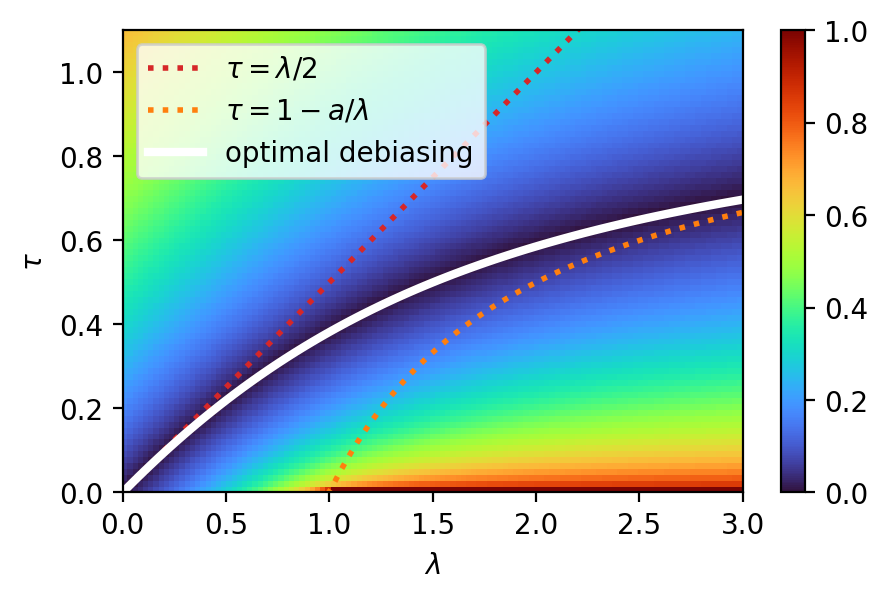}
\caption{Wasserstein distance $d^{-1}W_2(\mu^*_{\lambda,\tau},\mu^*_{0,0})=(\sqrt{b}-\sqrt{a})^2$ between $\mu^*_{\lambda,\tau}$ and $\mu^*_{0,0}$  from~\eqref{eq:gaussian-variance}, with $a=1.0$. The white line shows the best debiasing choice of $\tau^*(\lambda)$ from~\eqref{eq:nonasymp-debias} for which this distance is $0$. Below this line the barycenter is shrinked and above it is blurred. Dashed lines show the small and large $\lambda$ asymptotics of $\tau^*(\lambda)$.}
\end{figure}

\paragraph{Non-asymptotic debiasing} As can be seen from~\eqref{eq:nonasymp-debias}, the choice $\tau=\lambda/2$ gives the optimal debiasing only asymptotically as $\lambda\to 0$. For larger values of $\lambda$, this formula suggests to use a value for $\tau$ that is smaller than $\lambda/2$; and in any case smaller than $1$. Indeed, $\tau^*$ is concave and $\tau^*(\lambda)\leq \min \{\lambda/2,1\}$, which are its tangents at $0$ and $\infty$. Remark that for $\lambda$ large it holds $\tau^*(\lambda) = 1-a/\lambda +O((a/\lambda)^3)$. In practice, one may use the outer regularization value $\tau^*(\lambda)$  as a heuristic even for non-Gaussian measures, replacing $a$ by a notion of average of the variances of the marginals $\nu_k$.

\begin{table}
\centering
\begin{tabular}{l|l|l|c}
& Objective  & Variance (Isotropic Gaussian case) \\
\hline
$\mu_{0}^*$  &  $G_0$ &  $a$ \\
$\mu_{\lambda,0}^*$  &   $G_\lambda$ &  $\max\{a - \lambda, 0 \}$ \\
$\mu_{0,\tau}^*$ &  $G_0 +\tau H$ &   $a+2\tau-\tau^2/a+O(\tau^3/a^2)$\\
$\mu_{\lambda,\mathrm{div}}^*$ &  $G_\lambda -\frac12 T_\lambda(\cdot,\cdot)$  &   $a$\\
$\mu_{\lambda,\tau}^*$ &  $G_\lambda +\tau H$ &   $a+2\tau-\lambda+\tau(\lambda-\tau)/a +O((\lambda^3+\tau^3)/a^2)$\\
$\mu_{\lambda,\lambda/2}^*$ &  $G_\lambda +\frac{\lambda}{2} H$ &   $a+\lambda^2/(4a) +O(\lambda^3/a^2)$\\
 \hline

\end{tabular}
\caption{Variance of $(\lambda,\tau)$-barycenters when the $\nu_k$ are  Gaussian distributions with variance $aI_d$ (Eq.~\eqref{eq:gaussian-variance}).
}\label{tab:gaussian}
\end{table}

\paragraph{Barycenter of Dirac masses} Observe that for $a=0$, Eq.~\eqref{eq:gaussian-variance} gives $b=\tau$ (independently of $\lambda$). This can in fact be directly seen from the optimality conditions. Indeed, the Schr\"odinger system~\eqref{eq:schrodinger-system} gives that for each $k$,
$
\phi_{k}[{\mu^*_{\lambda,\tau}}](x) = \frac12 \Vert x-x_k\Vert^2 -\psi_k
$
 for some $\psi_k\in \RR$. Then the optimality condition~\eqref{eq:characterization} gives
 \begin{align*}
 \mu^*\propto e^{-\frac{\sum_k w_k \phi_{k}(x)}{\tau}} \propto e^{-\frac{\sum_k w_k \Vert x-x_k\Vert^2_2}{2\tau}}\propto e^{-\Vert x- \bar x\Vert^2/(2\tau)}.
 \end{align*}
The inner regularization $\lambda$ has no effect in this case because there is only one transport plan between any $\mu$ and each $\nu_k$ (in particular $\mu_{\lambda,0}^*=\mu^*_{0,0}$ for any $\lambda>0$).

 \subsection{Choice of the reference measure}\label{sec:reference-measure}
 As briefly mentioned in Section~\ref{sec:intro-douba}, there is an alternative formulation of $(\lambda,\tau)$-barycenters in terms of a change of reference measure in the definition of EOT. Let us justify this correspondence here. For a reference measure $\sigma \in \Mm_+(\Xx\times \Xx)$, let 
 \begin{align}\label{eq:EOT-ref}
T_\lambda(\mu,\nu|\sigma)& \coloneqq \min_{\gamma \in \Pi(\mu,\nu)} \int_{\Xx^2} c(x,y) \d\gamma(x,y) + \lambda \KL(\gamma | \sigma),\\
G_{\lambda,\alpha} (\mu)&\coloneqq \sum_{k=1}^K w_k T_{\lambda}(\mu,\nu_k|\mu^\alpha\otimes \nu_k).\label{eq:EOT-ref-2}
\end{align}
In the following discussion, $\mu^\alpha$ denotes $\mu$ itself when $\alpha=1$, the Lebesgue measure $\d x$ when $\alpha=0$ and $(\frac{\d \mu}{\d x})^\alpha\d x$ otherwise. In particular, it holds $G_{\lambda,1}=G_\lambda$. Note that one could equally choose the reference measures $\mu^\alpha\otimes \tilde \nu_k$ in Eq.~\eqref{eq:EOT-ref-2} for any $\tilde \nu_k$ such that $\KL(\nu_k|\tilde \nu_k)$ is finite without changing the barycenter. 

\begin{lemma}\label{lem:change-ref}
It holds
\begin{align*}
G_{\lambda,\alpha}(\mu) = G_\lambda(\mu)+\lambda(1-\alpha)H(\mu)
\end{align*}
and for $\tau=\lambda(1-\alpha)$ the minimizer of this functional is $\mu^*_{\lambda,\tau}$.
\end{lemma}
\begin{proof}
This is a consequence of the following property that can be found by direct computations: \rev{for any $\gamma\in \Pi(\mu,\nu)$ and $\tilde \mu,\tilde \nu \in \Pp(\Xx)$,
\begin{align}\label{eq:change-ref-measure}
\KL(\gamma|\tilde \mu\otimes \tilde \nu) = \KL(\gamma| \mu\otimes \nu) +\KL(\mu|\tilde \mu) + \KL(\nu|\tilde \nu).
\end{align}
See~\cite[Lem.~1.6]{marino2020optimal} for details, in particular for a proper treatment of the case of when one of the terms is $+\infty$.}
The result follows by taking $\tilde \nu=\nu=\nu_k$ and $\tilde \mu = \mu^\alpha$ for $\alpha\in \RR$, which leads, for $\gamma_k\in \Pi(\mu,\nu_k)$, to
\begin{gather*}
\KL(\gamma_k|\mu^\alpha\otimes \nu_k) = \KL(\gamma_k|\mu\otimes \nu_k) +\KL(\mu|\mu^\alpha) = \KL(\gamma|\mu\otimes \nu_k) +(1-\alpha)H(\mu).
\qedhere
\end{gather*}
\end{proof}

This lemma justifies the discussion after Def.~\ref{eq:optim-objective}. The effect of various choices of reference measure can be interpreted in view of the local expansion of Eq.~\eqref{eq:entropic-expansion} and are summarized in Table~\ref{table:reference-measures} (the two first rows correspond to cases discussed in~\cite{janati2020debiased} for Gaussian measures). In the table, we used reference measures which are symmetric in $\mu,\nu$ as is standard, to fix ideas.

To conclude this section, let us mention that it is in general more convenient to use the expression $F_{\lambda,\tau}=G_\lambda +\tau H$ as a sum of a smooth, convex functional and an entropy rather than $G_{\lambda,\alpha}$, although they are equivalent for $\tau=\lambda(1-\alpha)$. Our results in the next two sections rely on this decomposition.

\begin{table}
\centering
\begin{tabular}{c|l|l}
Reference measure $\sigma$ & Corresponding $\tau$ & $1$st order bias term \\
\hline
$\d x \otimes \d x$ &$ \tau=\lambda$ & $\frac{\lambda}{2}H(\mu)$ \\
$ \mu \otimes \nu$ & $\tau=0 $ & $- \frac{\lambda}{2} H(\mu) $  \\
$ \sqrt{\mu} \otimes \sqrt{\nu}$ &$\tau=\lambda/2$& $0$  \\
$\mu^\alpha\otimes \nu^\alpha$  &$ \tau=\lambda(1-\alpha)$& $\lambda(\frac12-\alpha)H(\mu)$  \\
\hline
\end{tabular}
\caption{Effect of the choice of the reference measure and interpretation as a $(\lambda,\tau)$-barycenter. The bias refers to the error $G_{\lambda,\alpha}(\mu)-G_0(\mu)$ when $c(x,y)=\frac12\Vert y-x\Vert^2_2$ (up to constant terms independent from $\mu$). In the last row, one needs $\alpha\leq 1$ so that $\tau\geq 0$.}\label{table:reference-measures}
\end{table}

\section{Stability and statistical estimation}\label{sec:statistics}

\subsection{General stability result}\label{sec:general-stability}
The goal of this subsection is to prove the following general stability result, where the main claim is (ii). In that statement, $W_1$ denotes the $L^1$-Wasserstein distance 
$$
W_1(\mu,\nu) \coloneqq \sup_{\mathrm{Lip}(f)\leq 1} \int_\Xx f(x)\d(\mu-\nu)(x),\quad \forall \mu,\nu \in \Pp(\Xx)
$$
where the supremum runs over the set of measurable functions $f:\Xx\to \RR$ with Lipschitz constant bounded by $1$, and $\dot H^{q}$ denotes the homogeneous Sobolev seminorm of order $q$ defined for $q\geq 0$ and a function $f\in L^1(\RR^d)$ with Fourier transform $\hat f$ by $\Vert f\Vert_{\dot H^q} \coloneqq \Vert \Vert \xi\Vert^{q} \hat f(\xi)\Vert_{L^2(\RR^d)}$ and $\Vert \mu-\nu \Vert_{\dot H^{-q}} \coloneqq \sup_{\Vert f\Vert_{\dot H^q}\leq 1}\int_{\Xx} f(x)\d(\mu-\nu)(x)$. The inhomogeneous Sobolev norms are analogously defined but with $\Vert f\Vert_{H^q} \coloneqq \Vert (1+\Vert \xi\Vert^2)^{q/2} \hat f(\xi)\Vert_{L^2(\RR^d)}$
\begin{assbox}
\begin{theorem}\label{thm:stability}
Let $(\nu_k)_{k=1}^K,(\hat \nu_k)_{k=1}^K \in \Pp(\Xx)^K$ and let $\hat \mu,\mu^*$ be their respective $(\lambda,\tau)$-barycenters with same weights $(w_k)_{k=1}^K$.
\begin{itemize}
\item[(i)] Assume that $c(x,\cdot)$ is $L$-Lipschitz for all $x\in \Xx$. Then for $\tau>0$ and $\lambda\geq 0$, it holds
$$
\KL(\hat \mu|\mu^*) \leq \frac{2L}{\tau}\sum_{k=1}^K w_k W_1(\nu_k,\hat \nu_k).
$$
\item[(ii)] Assume that $c\in \Cc^p(\Xx\times \Xx)$ for $p\geq 2$ and $\lambda,\tau>0$. Then there exists $C$  that depends on $c$ and $\Xx$ only such that
$$
\KL(\hat \mu|\mu^*) \leq \frac{C(1+\lambda^{1-p})}{\tau}\sum_{k=1}^K w_k \Vert \nu_k - \hat \nu_k\Vert_{\dot H^{-p}}.
$$
\end{itemize}
\end{theorem}
\end{assbox}
Note that the first bound, adapted from~\cite[Thm.~3.3]{bigot2019penalization} does not require $\lambda>0$ and only exploits the first-order regularity of the cost while with $\lambda>0$, the barycenter is able to exploit the higher order regularity of the cost to obtain stability under norms weaker than $W_1$. Another stability result in the literature is~\cite[Thm.~1]{theveneau2022stability} which proves $\ell_2$-stability of $(\lambda,\lambda)$-barycenters in a discrete setting, under $\ell_\infty$ perturbations of the cost matrix (with an exponential dependency in $\lambda^{-1}$).

Before we start the proof, let us state a lemma that only uses convexity of $G_\lambda$; the lower-bound is classical and the upper-bound is used later in Section~\ref{sec:dynamics}.
\begin{lemma}[Entropy sandwich] Let $\mu\in \Pp(\Xx)$ and consider the \emph{proximal Gibbs distribution} $\nu \propto \exp(-V[\mu]/\tau)$ where $V$ is the first-variation~\eqref{eq:first-variation} of $G_\lambda$. It holds
\begin{align*}
\tau \KL(\mu|\mu_{\lambda,\tau}) \leq F_{\lambda,\tau}(\mu)- F_{\lambda,\tau}(\mu^*_{\lambda,\tau})\leq \tau \KL(\mu|\nu).
\end{align*}
\end{lemma}
\begin{proof}
For brevity, let us write $F=F_{\lambda,\tau}$ and $\mu^*=\mu^*_{\lambda,\tau}$. By convexity of $G_\lambda$, we have
\begin{align*}
G_\lambda(\mu^*)+\int V[\mu^*]\d(\mu-\mu^*)\leq G_\lambda(\mu)\leq G_\lambda(\mu^*)+\int V[\mu]\d(\mu-\mu^*).
\end{align*}
Adding $\tau H(\mu)-G_\lambda(\mu^*)-\tau H(\mu^*)$ we get
\begin{gather*}
\int V[\mu^*]\d(\mu-\mu^*)+\tau H(\mu)-\tau H(\mu^*)\leq F(\mu)-F(\mu^*)\leq \int V[\mu]\d(\mu-\mu^*)+\tau H(\mu)-\tau H(\mu^*)
\intertext{which is equivalent to}
\tau \KL(\mu|\mu^*)-\tau \KL(\mu^*|\mu^*)\leq F(\mu)-F(\mu^*)\leq \tau \KL(\mu|\nu)-\tau \KL(\mu^*|\nu)
\end{gather*}
using $\mu^*\propto e^{-V[\mu^*]/\tau}$ (Thm.~\ref{thm:characterization}). The claim follows from $\KL(\mu^*|\mu^*)=0$ and $\KL(\mu^*|\nu)\geq 0$.
\end{proof}

\begin{proof}[Proof of Thm.~\ref{thm:stability}]
Let us start with an application of the previous lemma. For conciseness, let us drop the indices $\lambda,\tau$ and put hats on quantities defined using the measures $(\hat \nu_k)_{k=1}^K$ in place of $(\nu_k)_{k=1}^K$:
\begin{align*}
\tau \KL(\hat \mu  |\mu^* ) &\leq F (\hat \mu ) - F (\mu^* )\\
& = [F (\hat \mu ) - \hat F (\hat \mu )] + [\hat F (\hat\mu ) - \hat F (\mu^* )] + [\hat F (\mu ^*) - F (\mu^*)]\\
&\leq G_\lambda (\hat \mu ) - \hat G_\lambda (\hat \mu ) + 0 + \hat G_\lambda (\mu^* ) - G_\lambda (\mu^* )  \\
& \leq \sum_{k=1}^K w_k \big(T_\lambda(\hat \mu,\nu_k)-T_\lambda(\hat \mu,\hat \nu_k)\big) + \sum_{k=1}^K w_k\big(T_\lambda(\mu^*,\hat \nu_k)-T_\lambda(\mu^*,\nu_k)\big)\\
& \leq 2\sum_{k=1}^K w_k \sup_{\mu\in \Pp(\Xx)} \vert  T_\lambda(\mu, \hat \nu_k) - T_\lambda(\mu, \nu_k) \vert .
\end{align*}
Now by convexity of $T_\lambda$ in $\nu_k$, we have (denoting $\psi_{\mu,\nu}$ the Schr\"odinger potential from $\nu$ to $\mu$):
$$
\int \psi_{\mu,\hat \nu_k} \d [\nu-\hat \nu] \leq T_\lambda(\mu,\nu_k)-T_\lambda(\mu,\hat \nu_k) \leq \int \psi_{\mu,\nu_k} \d [\nu-\hat \nu].
$$
Given any norm $\Vert \cdot \Vert$ on the set of continuous functions defined up to constants, denoting $\Vert \sigma \Vert_*=\sup_{\Vert f\Vert\leq 1} \int f\d\sigma$ the dual norm on the space of signed measures with $0$ total mass, it follows 
$$
\tau \KL(\hat \mu|\mu^*) \leq 2 C \sum_{k=1}^K w_k \Vert \nu_k - \hat \nu_k\Vert_*
$$
where $C=\sup_{\mu,\nu\in \Pp(\Xx)}\Vert \psi_{\mu,\nu}\Vert$. Let us now consider the two claims separately.

\emph{(i)} If $c(x,\cdot)$ is $L$-Lipschitz for all $x\in \Xx$, then it can be seen from the Schr\"odinger system~\eqref{eq:schrodinger-system} that $\psi_{\mu,\nu}$ is also $L$-Lipschitz. Taking the Lipschitz semi-norm $\Vert f\Vert = \sup_{x\neq y}\frac{\vert f(x)-f(y)\vert}{\vert y-x\vert}$, the dual of which is the Kantorovich-Rubinstein norm $\Vert \nu_k -\hat \nu_k \Vert_* = W_1(\nu_k,\hat \nu_k)$ gives the first claim.

\emph{(ii)}  If $c\in \tilde \Cc^p$ then by differentiating the Schr\"odinger system $p$ times it follows that $\psi_{\mu,\nu}$ is $p$ times differentiable with $\Vert \psi_{\mu,\nu}\Vert_{\tilde \Cc^p} \leq C (1+\lambda^{1-p})$ where $C>0$ is a constant depending on the cost (Prop.~\ref{prop:propertiesEOT}-(iii)). Since $\Xx$ is assumed compact, it follows that the homogeneous Sobolev seminorm $\Vert \psi_{\mu,\nu}\Vert_{\dot H^p}$  admits the same bound (up to a constant depending on the diameter of $\Xx$). The conclusion follows from the fact that the seminorms $\dot H^p$ and $\dot H^{-p}$ are dual to each other.
\end{proof}

\subsection{Estimation from independent samples}
In this section, we consider the statistical properties $(\lambda,\tau)$-barycenters. Assume that we dispose of $n$ independent samples $x^{(k)}_1,\dots, x^{(k)}_n$ from each of the marginals $\nu_k$. Let $\hat \mu_\lambda\in \Pp_2(\Xx)$ be the plug-in estimator of the barycenter, defined as the $(\lambda,\tau)$-barycenter of the empirical marginals $\hat \nu_k = \frac1n \sum_{i=1}^n \delta_{x_i^{(k)}}$.

\begin{corollary}
Let $\tau,\lambda>0$, define $d'=2\lfloor d/2\rfloor$ and assume that $c\in \Cc^{1+d'/2}(\Xx\times \Xx)$. Let $\hat \mu_{\lambda,\tau}$ be the empirical barycenter and $\mu^*_{\lambda,\tau}$ the population barycenter. Then there is $C>0$ independent of $(\nu_k)_k$ such that 
\begin{align*}
\E [\KL(\hat \mu_{\lambda,\tau}|\mu^*_{\lambda,\tau})]\leq C \tau^{-1}(1+\lambda^{-d'/2})n^{-1/2}.
\end{align*}
\end{corollary}
\begin{proof}
We follow a strategy similar to the one used for the sample complexity of EOT~\cite{genevay2019sample}. In Thm.~\ref{thm:stability}, it is possible to replace the homogeneous Sobolev norm $\dot H^{-p}$ by the inhomogeneous norm $H^{-p}$. For $p=1+d'/2$, it is known that $H^{p}$ is a Reproducible Kernel Hilbert space norm, and by standard empirical process theory results~\cite{bartlett2002rademacher} one has 
\begin{equation*}
\E \Vert \hat \nu_k -\nu_k\Vert_{H^{-p}} \leq C n^{-1/2}. \qedhere
\end{equation*}
\end{proof}

\paragraph{Related work} To the best of our knowledge, this is the first estimation rate for an OT-like barycenter that does not suffer from the curse of dimensionality. The rate of estimation for $(0,\tau)$-barycenter was studied in~\cite{bigot2020statistical}, where the rate is cursed by the dimension because then the bound of Thm.~\ref{thm:stability}-(i) involves the quantity $W_1(\nu_k,\hat \nu_k)$ which is of order $n^{-1/d}$ for $d>2$~\cite{fournier2015rate}. In the same paper, they also studied $(\lambda,\lambda)$-barycenter but on a discrete space. The estimation of barycenters on discrete spaces is a rich topic but with a very different behavior~\cite{heinemann2023kantorovich, heinemann2022randomized}. See~\cite{panaretos2019statistical} for an introduction to statistical aspects of OT, including barycenters. We also note that there is a line of works (see e.g.~\cite{ahidar2020convergence, le2022fast}) that studies the different problem of estimation of barycenters given marginals $(\nu_k)_k$ sampled from a distribution in $\Pp(\Pp(\Xx))$ (as in~\eqref{eq:p-of-p}).

\section{Optimization with Noisy Particle Gradient Descent}\label{sec:dynamics}
We consider the computation of $(\lambda,\tau)$-barycenters when the marginals $(\nu_k)_k$ are discrete with $n$ atoms each. In this case, the size of the problem is given by $n$ (the number of atoms), $K$ (the number of marginals) and $d$ (the ambient dimension). 

For small scale problems where one of these quantities is small, several efficient algorithms exist. For instance when $n$ or $K$ is small, one can directly solve a linear program of size $O(n^K)$ -- the multimarginal formulation~\cite{agueh2011barycenters} -- to compute the $(0,0)$-barycenter. When $d$ is small, there exists efficient exact methods~\cite{altschuler2021wasserstein}. In that case, an alternative is to discretize the space, and use convex optimization algorithms to solve Eq.~\eqref{eq:optim-objective}. This includes approaches based on linear programming~\cite{anderes2016discrete, ge2019interior},  entropic regularization~\cite{cuturi2014fast, benamou2015iterative} or decentralized and randomized algorithms~\cite{ dvurechenskii2018decentralize,staib2017parallel,heinemann2022randomized}, see~\cite{peyre2019computational} for a review. This approach also applies for $(\lambda,\tau)$-barycenters (we compute 1D barycenters with this method in Section~\ref{sec:numerics}).

In this section, we focus on large scale problems ($n,d,K\gg 1$) where these discrete approaches are intractable. A stream of recent works proposed methods based on neural networks~\cite{korotin2020continuous, cohen2020estimating, li2020continuous, fan2020scalable}. These methods come with the advantages (useful statistical prior, reasonable iteration complexity) and the drawbacks (lack of optimization guarantees) of neural networks. Particle-based methods, which are closer in spirit to what follows, have been proposed such as fixed-point methods akin to Lloyd's algorithm~\cite{alvarez2016fixed, claici2018stochastic,von2022simple,backhoff2022stochastic} or~\cite[Chap.~5]{panaretos2020invitation} and a particle gradient method~\cite{daaloul2021sampling} for Wasserstein barycenters. Note that it is shown in~\cite{altschuler2022wasserstein} that Wasserstein barycenters are NP-hard to compute in large dimension. A Franck-Wolfe algorithm~\cite{luise2019sinkhorn} was proposed for Sinkhorn divergence barycenters.

In what follows, we propose a grid-free numerical method which is particularly well-suited to the structure of the problem of Eq.~\eqref{eq:optim-objective}, called Noisy Particle Gradient Descent (NPGD). We defer a detailed complexity analysis of this method to future works, and limit ourselves to an introduction of the algorithm with its exponential guarantee in the mean-field limit, which is an application of~\cite{chizat2022mean,nitanda2022convex}.

\subsection{Noisy Particle Gradient Descent} 
We parameterize the unknown measure as a mixture of $m\in \NN^*$ particles $\hat \mu=\frac1m \sum_{i=1}^m \delta_{X_j}$. Let $\mathbf{X} = (X_1,\dots,X_m)\in (\RR^d)^m$ encode the position of all particles and consider the function
\begin{equation}\label{eq:Gm}
G^{(m)}_\lambda(\mathbf{X}) \coloneqq G_\lambda\Big(\frac1m \sum_{i=1}^m \delta_{X_j}\Big).
\end{equation}
The NPGD algorithm we consider is simply a noisy gradient descent on $G^{(m)}_\lambda$, with an initialization sampled from some $\mu_0\in \Pp(\Xx)$. It is defined, for $\ell\in \NN$, as 
\begin{equation}\label{eq:NPGD}
\mathbf{X}[\ell+1] = \mathsf{P}_\Xx\big(\mathbf{X}[\ell] - m\eta \nabla G^{(m)}_\lambda(\mathbf{X}[\ell]) + \sqrt{2\eta\tau}   \mathbf{Z}[\ell]\big),\quad \mathbf{X}[0] \sim \mu_0^{\otimes m}
\end{equation}
where $\tau>0$ is the outer-regularization strength, $\eta>0$ is the step-size, $ \mathbf{Z}[1],  \mathbf{Z}[2],\dots$ are i.i.d.~standard Gaussian vectors and $ \mathsf{P}_\Xx$ is the Euclidean projection on $\Xx$. In the small step-size limit $\eta\to 0$ and setting $t=k\eta$, NPGD leads to a system of SDEs coupled via the empirical distribution of particles $\hat \mu_t$:
\begin{align}\label{eq:SDE}
\left\{
\begin{aligned}
\d X_i(t) &=  - \nabla V[\hat \mu_t](X_i(t))\d t + \sqrt{2\tau}   \d B_{t,i}  + \d\Phi_{t,i},\quad X_i(0) \sim \mu_0 \\
\hat \mu_t &=\frac1m \sum_{i=1}^m \delta_{X_i(t)}
\end{aligned}
\right.
\end{align}
where $(B_{t,i})_{t\geq 0}$ are independent Brownian motions in $\RR^d$, $\d\Phi_{t,i}$ is a boundary reflection (in the sense of Skorokhod problem) and $V[\mu]\in \Cc^1(\Xx)$ is the \emph{first-variation} of $G_\lambda$ at $\mu$ (see Definition~\ref{eq:first-variation}). The latter satisfies $\nabla V[\hat \mu](X_i)=m\nabla_{X_i} G^{(m)}_\lambda(\mathbf{X})$, hence~\eqref{eq:NPGD} is just the Euler-Maruyama discretization of~\eqref{eq:SDE} below. \rev{This expression also shows that, to compute $\nabla G^{(m)}_\lambda$ at each iteration of Eq.~\eqref{eq:NPGD}, one needs to compute the Schr\"odinger potentials between $\hat \mu$ and each $\nu_k$ (since they appear in the expression of $\nabla V[\hat \mu]$), which in turns requires to solve $K$ different EOT problems at each iteration.}

\subsection{Mean-Field Langevin dynamics} In the many-particle $m\to \infty$ limit, it can be shown that the particles behave like independent sample paths from the nonlinear SDE of McKean-Vlasov type:
\begin{equation}
\left\{
\begin{aligned}
\d X_t &=-\nabla V[\mu_t] (X(t))\d t + \sqrt{2\tau}\d B_t +\d\Phi_t, \quad X_0\sim \mu_0\\
\mu_t &=\mathrm{Law}(X_t)
\end{aligned}
\right.
\end{equation}
where $(B_t)_{t\geq 0}$ is a Brownian motion and $(\Phi_t)_{t\geq 0}$ a boundary reflection.
Moreover, the distribution $(\mu_t)_{t\geq 0}$ of particles solves the evolution equation
\begin{equation}\label{eq:PDE}
\partial_t \mu_t = \nabla \cdot \big(\mu_t \nabla V_\lambda[\mu_t] \big) +  \tau \Delta \mu_t
\end{equation}
starting from $\mu_0\in \Pp(\Xx)$ where $\nabla \cdot$ stands for the divergence operator. By solution of~\eqref{eq:PDE} here, we mean a curve $(\mu_t)_{t\geq 0}$ starting from $\mu_0$ that is absolutely continuous in Wasserstein space and satisfies Eq.~\eqref{eq:PDE} in the sense of distributions  with no-flux boundary conditions. 

This drift-diffusion equation is an instance of \emph{Mean-Field Langevin} dynamics, a class of drift-diffusion dynamics  studied in~\cite{mei2018mean, hu2019mean, nitanda2022convex,chizat2022mean} for convex $G_\lambda$. It can be interpreted as the gradient flow of the functional $F_{\lambda,\tau}$ of Eq.~\eqref{eq:optim-objective} under the $W_2$ Wasserstein metric. 

In our theoretical analysis, we focus on the analysis of the mean-field limit~\eqref{eq:PDE}. For quantitative convergence results in the many-particle $m\to \infty$ and small step-size $\eta\to 0$ limits, one can refer to the classical work~\cite{sznitman1991topics}. The particular case of reflecting boundary conditions has been treated in~\cite{javanmard2020analysis}, following earlier works on the analysis of SDEs with reflection~\cite{tanaka1979stochastic, lions1984stochastic}.

\subsection{Exponential convergence}
It has been shown independently in~\cite{nitanda2022convex} and~\cite{chizat2022mean} that if the function $G_\lambda$ is convex, regular enough and that a certain family of log-Sobolev inequalities holds, then dynamics of the form Eq.~\eqref{eq:PDE} converge at an exponential rate  to the unique minimizer of $F_{\lambda}$. Let us apply this result in our context, where this leads to a dynamics that converges to the $(\lambda,\tau)$-barycenter.
\begin{assbox}
\begin{theorem}\label{thm:global-convergence}
Assume that $c\in \Cc^2(\Xx)$ (where, we recall, $\Xx$ is compact) and let $\lambda,\tau>0$. Then there exists a unique solution to~\eqref{eq:PDE}. Moreover, there exists $\rho_\tau>0$ such that if $\mu_0\in \Pp(\Xx)$ is such that $F_{\lambda,\tau}(\mu_0)<\infty$, then it holds 
\begin{align*}
\tau \KL(\mu_t|\mu^*_{\lambda,\tau})\leq F_{\lambda,\tau}(\mu_t) - F_{\lambda,\tau}(\mu^*_{\lambda,\tau})\leq e^{-\rho_\tau t} \big( F_{\lambda,\tau}(\mu_0) - F_{\lambda,\tau}(\mu^*_{\lambda,\tau}) \big).
\end{align*}
\end{theorem}
\end{assbox}
\begin{proof}
We have semi-convexity of $F_{\lambda,\tau}$ along Wasserstein geodesics, by \cite[Thm.~4.1]{carlier2024displacement} for the first component $G_\lambda$ and by a standard result~\cite{santambrogio2015optimal} for the  $H$ component. Thus the general well-posedness results from~\cite{ambrosio2005gradient} applies. For the exponential convergence -- in function value and in relative entropy -- we apply the result from~\cite[Thm. 3.2]{chizat2022mean}, see also~\cite{nitanda2022convex} (although stated on $\RR^d$, the argument goes through on a compact domain). 

The main assumptions to check are that (i) $\mu\mapsto G_\lambda(\mu)$ is convex (Prop.~\ref{prop:Reg_G}), (ii) that a global minimizer $\mu^*_{\lambda,\tau}$ exists (Thm.~\ref{thm:characterization}) and finally we need to check that the probability measure $\hat \mu_t \propto e^{-V[\mu_t]} \in \Pp(\Xx)$ satisfies a log-Sobolev inequality, uniformly in $t$ (Assumption 3 in~\cite{chizat2022mean}). 

Since $\Xx$ is bounded, the normalized Lebesgue measure satisfies a log-Sobolev inequality~\cite[Thm.~7.3]{ledoux1999concentration}. By the Holley-Stroock perturbation criterion~\cite{holley1986logarithmic} $\hat \mu_t$ satisfies it as well; this criterion applies here because $\sup_x V[\mu](x)- \inf_x V[\mu](x)$ is bounded, uniformly in $\mu\in \Pp(\Xx)$ by Prop.~\ref{prop:Reg_G}. 
\end{proof}
The contraction rate $\rho_\tau$ is of the form $\tau e^{-L\cdot \mathrm{diam}(\Xx)/\tau}$ where $L$ is the Lipschitz constant of $c$ (and $L\cdot \mathrm{diam}(\Xx)$ is a uniform bound on $\Vert V[\mu]\Vert_{\mathrm{osc}}$). It thus approaches $1$ exponentially fast as $\tau$ decreases. It could be of interest to exhibit settings where this dependency in $\tau$ is milder. We also note that both inner and outer regularizations are needed to obtain Thm.~\ref{thm:global-convergence}: in particular the inner-regularization is necessary to obtain well-posedness of the PDE~\eqref{eq:PDE}.

\section{Numerical experiments}\label{sec:numerics}
The (Julia) code to reproduce the experiments is available online\footnote{\url{https://github.com/lchizat/2023-doubly-entropic-barycenter.git}}.

\paragraph{Comparison of 1D-barycenters} On Fig.~\ref{fig:comparison-1D} we compare various barycenters for $K=3$ and $(\nu_k)_{k=1}^3$ probability densities on $\Xx=[0,1]$ with cost $c(x,y)=\nicefrac12 \Vert y-x\Vert^2_2$. The $(\lambda,\tau)$-barycenters have been computed numerically using gradient ascent on the dual problem~\eqref{eq:dual} after discretizing the problem on a regular grid of size $200$ and replacing $H$ by the negative discrete entropy. We observed that the algorithm converged linearly for small enough step-sizes. For reference, we plot in blue the unregularized Wasserstein barycenter $\mu^*_{0,0}$ that is computed by taking the $L^2$-barycenter of the quantile functions~\cite[Chap.~2]{santambrogio2015optimal}. We observe that the choice $\tau=\lambda/2$ indeed gives, visually, the best approximation of $\mu^*_{0,0}$ for fixed $\lambda$. For comparison, we also plot the Sinkhorn divergence barycenter $\mu^*_{\lambda,\mathrm{div}}$, which is computed with~\cite[Alg.~1]{janati2020debiased}. We observe that it also approaches weakly $\mu^*_{0,0}$ but its density displays strong oscillations, which may suggest that this object is in general less well-behaved than $(\lambda,\tau)$-barycenters.

\begin{figure}
\centering
\begin{subfigure}{0.33\linewidth}
\centering
\includegraphics[scale=0.4]{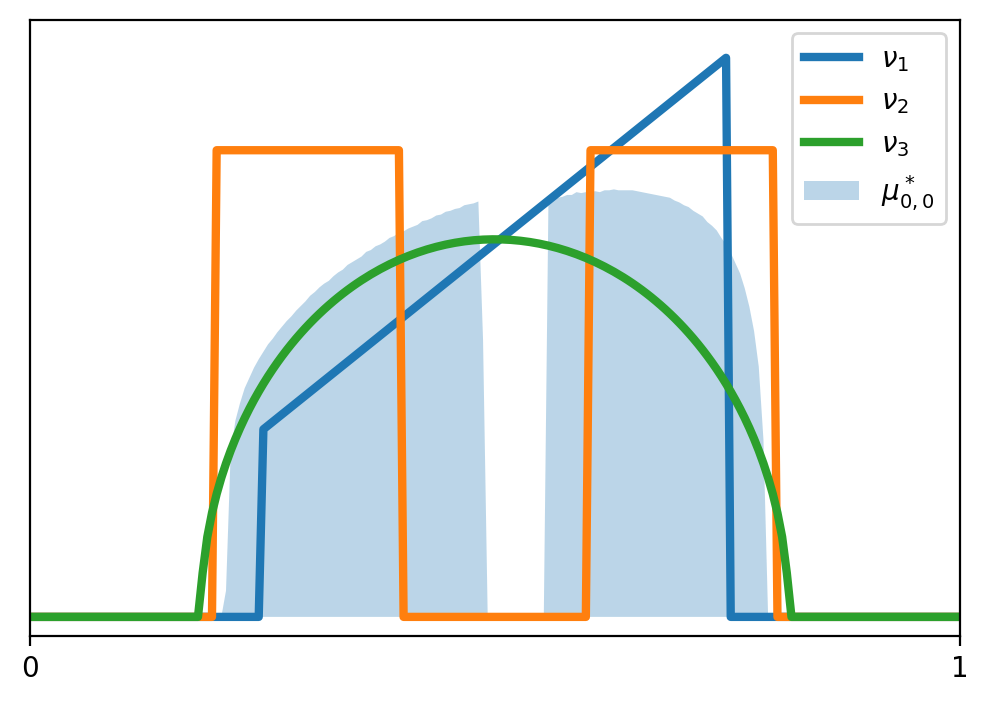}
\caption{Densities $(\nu_k)_{k=1}^3$ and $\mu^*_{0,0}$}\label{subfig:1D-marg}
\end{subfigure}%
\begin{subfigure}{0.33\linewidth}
\centering
\includegraphics[scale=0.4]{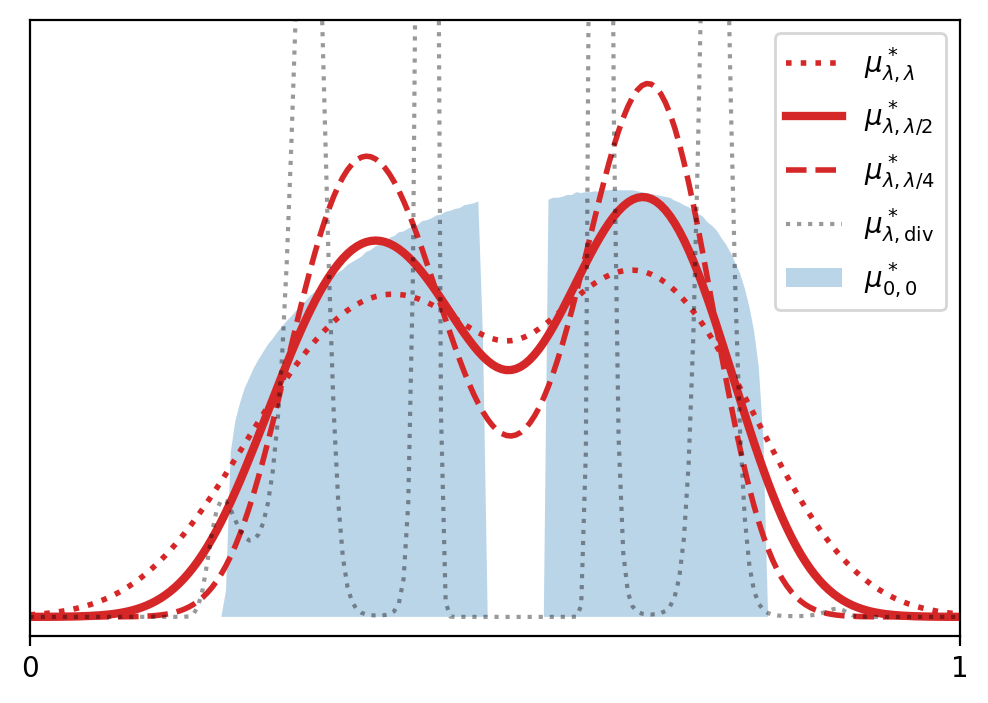}
\caption{$\lambda=1/128$}\label{subfig:1D-128}
\end{subfigure}%
\begin{subfigure}{0.33\linewidth}
\centering
\includegraphics[scale=0.4]{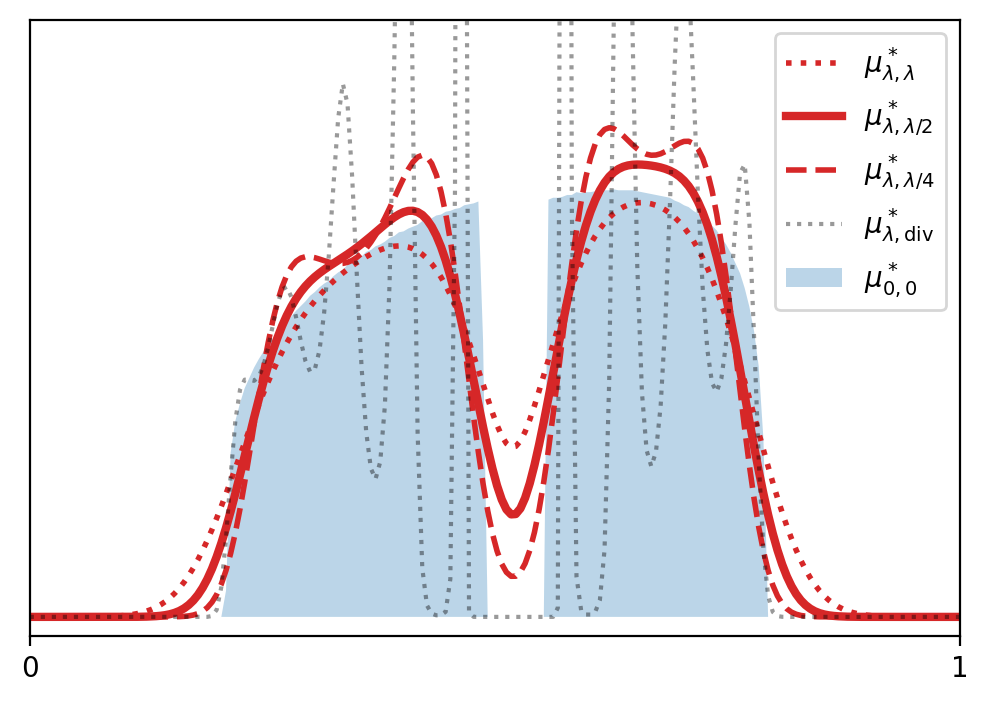}
\caption{$\lambda=1/512$}\label{subfig:1D-512}
\end{subfigure}
\caption{Iso-barycenters for $(\nu_k)_{k=1}^3$ probability densities on $\Xx=[0,1]$ (displayed on the left). For $\lambda,\tau>0$, $\mu^*_{\lambda,\tau}$ is computed with gradient ascent on the discretized dual problem.}
\label{fig:comparison-1D}
\end{figure}

\paragraph{Escaping stationary points with noisy particle gradient descent} To illustrate the global convergence of Noisy Particle Gradient Descent (NPGD) to $\mu^*_{\lambda,\tau}$ in the mean-field limit, we consider a configuration for which $G_0$ has a ``bad'' local minimum and initialize the dynamics at this measure as shown on Fig.~\ref{subfig:2D-1}. We consider on $\Xx=\RR^2$ the barycenter of\footnote{This example was suggested to us by Hugo Lavenant.} 
\begin{align*}
\nu_1 = \frac12 (\delta_{(-3,0)}+ \delta_{(-1,1)})
&&\text{and}&& \nu_2 = \frac12 (\delta_{(1,1)}+ \delta_{(3,1)})
\end{align*}
with cost $c(x,y)=\frac12 \Vert y-x\Vert^2_2$. It can be checked with direct computations that $\mu^*_{(0,0)} = \frac12 (\delta_{(-1,\nicefrac12)}+ \delta_{(1,\nicefrac12)})$ and that the measure $\mu^*_{\mathrm{init}} =\frac12 (\delta_{(0,0)}+ \delta_{(0,1)})$ is a stable local minimizer (in the sense that when parameterized by the positions of these two Dirac masses, $G_0$ is locally minimized by $((0,0),(0,1))$ and its Hessian at this point is positive definite, proportional to the identity).
In particular, fixed-point or gradient descent iterations for $G_0$ would not move away from this stationary point. On Fig.~\ref{subfig:2D-1} we observe, in accordance to Thm.~\ref{thm:global-convergence}, that NPGD can escape from the neighborhood of $\mu_{\mathrm{init}}$ and converges to a discrete approximation of $\mu^*_{(\lambda,\lambda/2)}$, itself an approximation of $\mu^*_{(0,0)}$. Note that this is only observed when $\tau$ is large enough, otherwise the dynamics might be trapped in ``metastable'' states (here $\lambda=0.1$ and $\tau=\lambda/2$). We have used $m=200$ particles, a step-size of $0.5$ and Fig.~\ref{subfig:2D-400} represents the state of NPGD after $400$ iterations.

\begin{figure}
\centering
\begin{subfigure}{1.0\linewidth}
\centering
\includegraphics[scale=0.8]{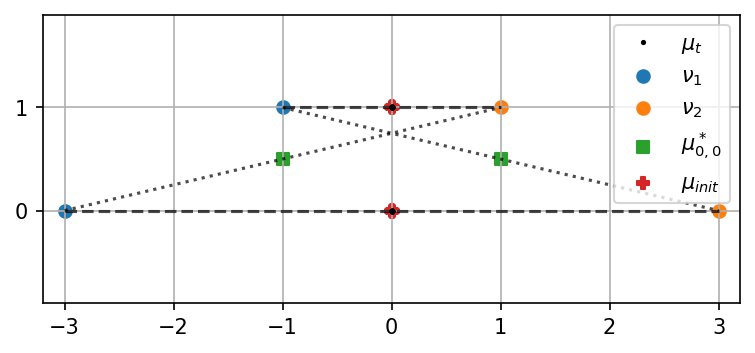}
\caption{Initializing on a local minimum of $G_0$ ($t=0$)}\label{subfig:2D-1}
\end{subfigure}\\
\begin{subfigure}{1.0\linewidth}
\centering
\includegraphics[scale=0.8]{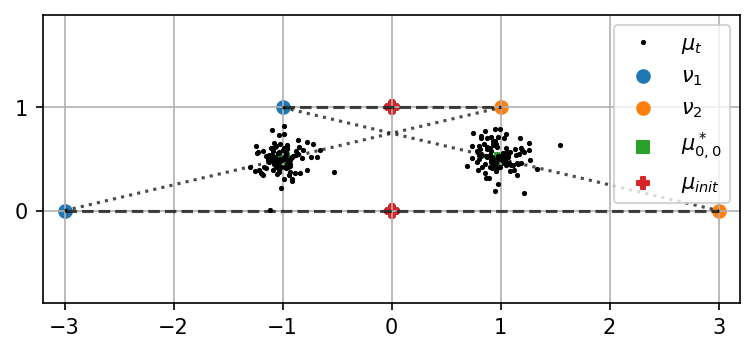}
\caption{Approximate global minimizer ($t$ large)}
\label{subfig:2D-400}
\end{subfigure}
\caption{Illustration of the global convergence of the noisy particle gradient descent algorithm. (left) We initialize $m=200$ particles on the measure $\mu_{\mathrm{init}}$ which is a local minimizer of $G_0$. (right) The algorithm converges to a distribution which weakly approximates $\mu^*_{\lambda,\lambda/2}$, which is itself an approximation of $\mu^*_{0,0}$.}
\label{fig:NPGD}
\end{figure}

\section{Conclusion}
We have proposed doubly-regularized EOT barycenters, a formulation of barycenters over $\Pp(\Xx)$ that has several benefits: (i) it gathers most previously studied OT-like barycenters in a single framework, (ii) it leads to several desirable properties such as regularity, stability and approximation of Wasserstein barycenters (for $\tau\sim \lambda/2$), and (iii) it is amenable to both convex optimization methods and grid-free methods with global convergence guarantees. 

In this paper, we have covered a diverse range of theoretical topics (regularity, approximation, statistics, optimization) as it is really all these aspects that together show the trade-offs at play in the choice of regularization. In a separate paper~\cite{vavskevivcius2023computational}, we have studied the computational complexity of $(\lambda,\tau)$-barycenters via Sinkhorn-like algorithms. In future works, it would be desirable to develop in depth analyses of other specialized properties of these objects.



\printbibliography
\appendix

\section{Closed-forms for isotropic Gaussian distributions}\label{app:gaussians}
In this section, we obtain closed form expressions for the $\mu^*_{\lambda,\tau}$ barycenter in the simple setting of centered isotropic Gaussian probability measures, where we can rely on closed-forms expressions for entropic optimal transport. We consider the following problem:
\begin{align}\label{eq:simplified}
\min_{\nu\in \Pp(\RR^d)} T_\lambda(\mu,\nu) + \tau H(\nu)
\end{align}
for $\mu=\Nn(0,a)$ where $a>0$ is the variance. This is enough to recover the general case of barycenters of isotropic Gaussian distributions with identical variances since the effect of having non-centered Gaussian distributions can be factored out as in~\cite{janati2020entropic}. Note that~\cite{ballu2020stochastic} also considered this problem but did not obtain a closed-form solution.  Let us look for a Gaussian solution of the form $\nu = \Nn(0,b)$. By~\cite{janati2020entropic}, we have that the optimal dual potentials are $\varphi(x)=\frac12 \tilde u \Vert x\Vert_2^2$ and $\psi(y)=\frac12 \tilde v \Vert y\Vert^2$ with
\begin{align*}
\tilde u &= 1 - \frac{2b}{\xi(b)+\lambda},&
\tilde v &= 1 - \frac{2a}{\xi(b)+\lambda},&
\xi(b) &=  \sqrt{4ab+\lambda^2}
\end{align*}
where the variables from~\cite{janati2020entropic} and our variables are related by $\sigma^2 \rightarrow \lambda$, $U \rightarrow-\lambda \tilde u I$, $V \rightarrow -\lambda \tilde v I$, $C \rightarrow I\cdot (\xi(b)-\lambda)/2$. Also the EOT cost and entropy is
\begin{align*}
\frac{2}{d}T_\lambda(\mu,\nu) &= a+b - \xi(b) +\lambda \log(\xi(b)+\lambda)+C,& \frac{2}{d}\tau H(\mu) &=-\tau \log(b)+C
\end{align*}
where $C$ denotes a constant independent of $a$ and $b$. When restricted to $\nu$ of the form $\Nn(0,b)$, the optimization problem thus becomes
$$
\min_{b>0} a+b - \xi(b) +\lambda \log(\xi(b)+\lambda)-\tau \log(b).
$$
The first order optimality condition gives the implicit equation
\begin{align}\label{eq:sol-gaussian}
1- \frac{2a}{\xi(b)+\lambda}-\frac{\tau}{b}=0.
\end{align}
Now, let us argue that if $b$ satisfies this equation, then $\nu$ is in fact the unique minimizer of~\eqref{eq:simplified}. Indeed, \eqref{eq:simplified} admits a unique solution characterized by the first order optimality condition (which can be justified, in this non-compact setting, as in~\cite{janati2020entropic}):
$$
\log\Big(\frac{\d\nu}{\d x}\Big) +\frac{\psi_{\mu,\nu}}{\tau} = C 
$$
But if $\nu=\Nn(0,b)$ with $b$ solution to~\eqref{eq:sol-gaussian} then it holds
\begin{align}
\tilde v &=\tau/b&\Rightarrow&&\frac{\psi_{\mu,\nu}(y)}{\tau} - \frac{\Vert y\Vert^2_2}{2b}=0.
\end{align}
\rev{In other words, if $b$ is solution to~\eqref{eq:sol-gaussian} then $\nu$ satisfies the necessary and sufficient optimality conditions of Thm.~\ref{thm:characterization}, which proves that $\nu$ is indeed the unique minimizer of $F^*_{\lambda,\tau}$.}
\paragraph{Explicit solution.} Using $b=(\xi(b)^2-\lambda^2)/(4a)$, let us express~\eqref{eq:sol-gaussian} in terms of $\xi=\xi(b)$:
\begin{align*}
b(\xi+\lambda)-2ab-\tau(\xi+\lambda)=0 \Leftrightarrow \frac{\xi+\lambda}{4a}(\xi^2-2a\xi-\lambda^2+2a\lambda-4a\tau)=0.
\end{align*}
Solving for $\xi$ and taking the largest solution (which must be the unique solution leading to $b\geq 0$) leads to
\begin{align*}
\xi = a+\sqrt{(a-\lambda)^2+4a\tau} &&\Rightarrow&& b = \frac{\big(a+\sqrt{(a-\lambda)^2+4a\tau}\big)^2-\lambda^2}{4a}.
\end{align*}
Let us now consider some particular cases:
\begin{itemize}
\item In the limit $\tau\to 0$, we have
$$
b= \frac{(a+\vert a-\lambda\vert)^2-\lambda^2}{4a} =
\begin{cases}
a-\lambda &\text{if $a>\lambda$}\\
0 &\text{if $a\leq \lambda$}
\end{cases}
$$
\item In the limit $\lambda\to0$, we have
$$
b = \frac{a}{4}\big( 1+\sqrt{1+4\tau/a}\big)^2 = a+2\tau+O(\tau^2)
$$
\item for $\tau=\lambda$, we have
$$
b= \frac{(a+\vert a+\lambda\vert)^2-\lambda^2}{4a} = a +\lambda
$$
\item for $\tau=\lambda/2$, we have
\begin{align*}
b &=  \frac{(a+\sqrt{ a^2 +\lambda^2})^2-\lambda^2}{4a} \\
&= \frac{a^2+a^2+\lambda^2+2a\sqrt{ a^2 +\lambda^2}-\lambda^2}{4a} \\
&=\frac{2a^2+2a^2\sqrt{ 1 +(\lambda/a)^2}}{4a}\\
&=\frac{a}{2}(1+\sqrt{1+(\lambda/a)^2}) =a+\frac{\lambda^2}{4a}+O(\lambda^4)
\end{align*}
\end{itemize}

\paragraph{Best choice of $\tau$} Let us now solve for $\tau$ such that $b=a$, which we know is asymptotically $\tau\sim \lambda/2$. It holds
\begin{align*}
b-a = 0 &&\Leftrightarrow && (a+\sqrt{(a-\lambda)^2+4a\tau})^2-\lambda^2 -4a^2 = 0.
\end{align*} 
Let us define the intermediate unknown $\chi = \sqrt{(a-\lambda)^2+4a\tau} \Leftrightarrow \tau = (\chi^2-(a-\lambda)^2)/(4a)$. It follows
\begin{align*}
(a+\chi)^2-\lambda^2-4a^2 =0 && \Leftrightarrow &&  \chi^2 +2a\chi -\lambda^2-3a^2=0.
\end{align*} 
Solving for $\chi$, with the constraint $\chi\geq 0$ gives
\begin{align*}
\chi=\sqrt{4a^2+\lambda^2}-a&& \Leftrightarrow && \tau = \frac{(\sqrt{4a^2+\lambda^2}-a)^2-(a-\lambda)^2}{4a}.
\end{align*} 
Developing the squares and rearranging we get
\begin{align*}
\tau &= \frac{4a^2+-2a\sqrt{4a^2+\lambda^2}+2a\lambda}{4a} \\
&= \frac{\lambda}{2} + a\Big(1-\Big(1+\frac{\lambda^2}{4a^2}\Big)^{1/2}\Big)\\
&=\frac{\lambda}{2} -\frac{\lambda^2}{8a} +O(\lambda^4/a^3).
\end{align*}

\section{Optimality conditions}
In this section, we reproduce the classical argument for the optimality conditions of Thm.~\ref{thm:characterization}, in a  setting more general than in Section~\ref{sec:regularity} since the exact structure of the functional does not play a role. Consider the following assumptions:
\begin{itemize}
\item $G:\Pp(\Xx)\to \RR$ is a convex and weakly continuous functional that admits $V[\mu]$ as a first-variation, in the sense that 
\begin{align}
\forall \mu,\tilde \mu \in \Pp(\Xx),\; \lim_{\epsilon\, \downarrow\, 0}\frac{1}{\epsilon} \Big(G((1-\epsilon)\mu+\epsilon \tilde \mu) - G(\mu)\Big) = \int_\Xx V[\mu](x)\d (\tilde \mu-\mu)(x) .
\end{align}
\item $\pi\in \Pp(\Xx)$ is a reference probability measure (such as the normalized Lebesgue measure on $\Xx$ in the main text).
\end{itemize}
The following is a direct adaptation of the arguments in~\cite[Prop.~8.7]{santambrogio2015optimal}, which we reproduce below for convenience.
\begin{proposition}\label{prop:general-min-entropy-pb}
For $\tau>0$, the functional $F=G+ \KL(\cdot|\pi)$ admits a unique minimizer $\mu^*$ in $\Pp(\Xx)$ which is absolutely continuous with respect to $\pi$ and satisfies
\begin{align}\label{eq:opti-cond}
\frac{\d \mu^*}{\d \pi}(x) = \exp(\chi -V[\mu^*](x)) &&\text{with}&& \chi = -\log \int \exp(-V[\mu^*](x))\d\pi(x).
\end{align}
Moreover, any $\mu\in \Pp(\Xx)$ that is absolutely continuous and satisfies Eq.~\eqref{eq:opti-cond} is equal to $\mu^*$.
\end{proposition}
Since the EOT barycenter functional $G_\lambda$ satisfies these assumptions (by Prop.~\ref{prop:Reg_G}), this directly implies the optimality condition of Thm.~\ref{thm:characterization} by taking $\pi$ the normalized Lebesgue measure.
\begin{proof}
The functional $G$ is weakly continuous and $\KL(\cdot|\pi)$ is weakly lower-semicontinuous~\cite[Sec.~7.1.2]{santambrogio2015optimal} so $F$ is weakly lower-semicontinuous. It is not identically $+\infty$ since $F(\pi)<+\infty$. Since $\Pp(\Xx)$ is weakly compact, there exists at least one minimizer $\mu^*$ by the direct method of the calculus of variations. Moreover, since $\KL(\cdot| \pi)$ is strictly convex, $F$ is strictly convex and the minimizer is unique and since $\KL(\mu^*|\pi)<\infty$ we have $\mu^*$ absolutely continuous with respect to $\pi$ and let $\rho^*$ denotes its density.

We now prove the optimality condition, by considering a perturbation $\mu_\epsilon = (1-\epsilon)\mu^*+\epsilon \tilde \mu$ of density $\rho_\epsilon = (1-\epsilon)\rho^*+\epsilon \tilde \rho$ where $\tilde \mu \in \Pp(\Xx)$ is of the form $\tilde \mu = \tilde \rho \pi$ with $\tilde \rho \in L^\infty(\pi)$, that is there exists $M>0$ such that $\vert \tilde \rho (x)\vert \leq M$ for $\pi$ almost every $x$. On the one hand, we have by definition of the first-variation that
$$
\frac{\d}{\d \epsilon} G(\mu_\epsilon)\vert_{\epsilon=0} = \int V[\mu]\d (\tilde \mu -\mu^*).
$$
On the other hand for the relative entropy term $\KL(\cdot|\pi)$, the integrand can be differentiated in $\epsilon$ pointwisely thus giving $\log \rho_\epsilon(x)(\tilde \rho-\rho^*)$. For $\epsilon<1/2$, one can check that these functions are dominated by $(\rho^*+M)(\vert \rho^*\vert +\log M)$ which is in $L^1(\pi)$, owing to the fact that $\rho^*, \log \rho^*,\rho^*\log\rho^* \in L^1(\pi)$ by Lem.~\ref{lem:positive-minimizer}. This allows to differentiate under the integral sign and proves that
$$
\frac{\d}{\d \epsilon} \KL(\mu_\epsilon| \pi)\vert_{\epsilon=0} = \int \log \rho^* \d (\tilde \mu -\mu^*).
$$
Now by optimality of $\mu^*$ we must have, denoting $g^*=V[\mu^*]+\log \rho^*$
$$
0 \leq \frac{\d}{\d \epsilon} F(\mu_\epsilon) = \int g^* \d (\tilde \mu -\mu^*) \quad\Rightarrow\quad \int g^*\d\tilde \mu \leq \int g^*\d\mu^*.
$$
This implies that $\int g^*\d\mu^*$ equals the $\pi$-essential supremum of $g^*$ since otherwise, one could build a perturbation $\tilde \mu$ concentrating on a superlevel of $g^*$ which would contradict this inequality. But since $\rho^*>0$ (Lem.~\ref{lem:positive-minimizer}) we conclude that $g^*=\chi$ $\pi$-a.e. for some $\chi \in \RR$. The value of $\chi$ is determined by the constraint $\mu^* \in \Pp(\Xx)$.
Finally, any $\mu\in \Pp(\Xx)$ that is absolutely continuous and satisfy~\eqref{eq:opti-cond} is equal to $\mu^*$, because it satisfies the first order optimality conditions of a strictly convex problem.
\end{proof}

\begin{lemma}\label{lem:positive-minimizer}
Any minimizer $\mu^*$ of $F$ is absolutely continuous with respect to $\pi$ and its density $\rho^*=\frac{\d \mu^*}{\d \pi}$ must satisfy $\rho^*>0$ and $\log \rho^*\in L^1(\pi)$.
\end{lemma}
\begin{proof}
For $\epsilon\in [0,1]$, let us define $\rho_\epsilon = (1-\epsilon)\rho^*+\epsilon$, $\mu_\epsilon=\rho_\epsilon \pi$ and compare $F(\mu^*)$ to $F(\mu_\epsilon)$. By optimality of $\mu^*$ and convexity of $G$, we may write
$$
\KL(\mu^*|\pi)- \KL(\mu_\epsilon|\pi) \leq G(\mu_\epsilon) - G(\mu^*)\leq (1-\epsilon) G(\mu^*) +\epsilon G(\pi) - G(\mu^*) = \epsilon (G(\pi)-G(\mu^*)).
$$
We get 
$$
\int ( f(\rho^*)-f(\rho_\epsilon))\d\pi \leq C\epsilon
$$
where $f(t)=t\log(t)-t+1$ (and $f(0)=1$ by convention). Write 
\begin{align*}
A = \{x\in \Xx\;;\; \rho^*(x)>0\}, && B = \{x\in \Xx\;;\; \rho^*(x)=0\}.
\end{align*}
Since $f$ is convex we have for $x\in A$
$$
f(\rho^*(x))-f(\rho_\epsilon(x)) \geq ((\rho^*(x)-\rho_\epsilon(x))f'(\rho_\epsilon(x)) = \epsilon (\rho^*(x)-1)\log \rho_\epsilon(x).
$$
For $x\in B$, we simply write $f(\rho^*(x))-f(\rho_\epsilon(x)) =f(0)-f(\epsilon)=-\epsilon \log \epsilon -\epsilon$. This allows to write
$$
-\epsilon (\log \epsilon +1)\pi(B) +\epsilon  \int_{A} (\rho^*(x)-1)\log \rho_\epsilon(x) \d\pi(x) \leq C\epsilon
$$
and, dividing by $\epsilon$,
\begin{align}\label{eq:proof-positive}
-(\log \epsilon +1)\pi(B) +  \int_{A} (\rho^*(x)-1)\log \rho_\epsilon(x) \d\pi(x) \leq C.
\end{align}
Note that we always have $(\rho^*(x)-1)\log \rho_\epsilon(x)\geq 0$ (just distinguish between the case $\rho^*(x)\geq 1$ and $\rho^*(x)\leq 1$). Thus, we may write
$$
-(\log \epsilon +1)\pi(B) \leq C.
$$
Letting $\epsilon\to 0$ proves $\pi(B)=0$, hence $\rho^*(x)>0$, $\pi$-almost everywhere.

We now come back to Eq.~\eqref{eq:proof-positive} which is an upper-bound on the nonnegative functions $(\rho^*(x)-1)\log \rho_\epsilon(x)$. By Fatou's lemma we have as $\epsilon\to 0$
$$
\int_{\Xx} (\rho^*(x)-1)\log \rho^*(x) \d\pi(x) \leq C.
$$
Since this is the integral of a nonnegative function, it follows that $(\rho^*-1)\log \rho^*$ is in $L^1(\pi)$. Since we already know that $\rho^*\log \rho^*\in L^1(\pi)$, it follows $\log \rho^* \in L^1(\pi)$.
\end{proof}

\begin{lemma}[Strong convexity of $H$]\label{lem:H-strongly-convex}
The functional $H:\Pp(\Xx)\to \RR\sup \{+\infty\}$ is $1$-strongly convex relative to the total variation norm.
\end{lemma}
\begin{proof}
Let $\mu,\nu\in \Pp(\Xx)$ such that $H(\mu),H(\nu)<+\infty$ (note that since $\Xx$ is compact, $H$ is lower-bounded as then this function is equal, up to a constant, to the Kullback-Leibler divergence with respect to the normalized Lebesgue measure, itself lower-bounded by $0$). Let $\lambda\in {]0,1[}$ and let $\sigma=\lambda \mu+(1-\lambda)\nu$ which also satisfies $H(\sigma)<+\infty$ by convexity of $H$ and $\KL(\mu|\sigma)\leq \log(1/\lambda)<+\infty$. Moreover, $\int \log\Big(\frac{\d\sigma}{\d x}\Big)\d\mu$ has a well defined value in $\RR \cup {+\infty}$ since the negative part of the integrand can be upper bounded by $\max\{-\log(\lambda \mu),0\}$ which is integrable under $\mu$. Therefore under these assumptions, we can write
$$
\KL(\mu|\sigma) = H(\mu)-H(\sigma) -\int_\Xx \log\Big(\frac{\d\sigma}{\d x}\Big)\d(\mu-\sigma)
$$
and all terms are finite (an analogous formula holds for $\KL(\nu|\sigma)$). By Pinsker's inequality, it holds
\begin{align*}
\lambda \KL(\mu|\sigma) + (1-\lambda) \KL(\nu|\sigma) &\geq \frac{\lambda}{2} \Vert \mu-\sigma\Vert^2_\TV+ \frac{1-\lambda}{2} \Vert \nu-\sigma\Vert^2_\TV = \frac{\lambda(1-\lambda)}{2}\Vert \mu-\nu\Vert^2_\TV.
\end{align*}
It follows
\begin{align*}
\frac{\lambda(1-\lambda)}{2}\Vert \mu-\nu\Vert^2_\TV&\leq \lambda \Big(H(\mu)-H(\sigma) -\int_\Xx \log\Big(\frac{\d\sigma}{\d x}\Big)\d(\mu-\sigma)\Big)\\
&\quad + (1-\lambda) \Big(H(\nu)-H(\sigma) -\int_\Xx \log\Big(\frac{\d\sigma}{\d x}\Big)\d(\nu-\sigma)\Big)\\
&= \lambda H(\mu) +(1-\lambda) H(\nu) - H(\sigma)
\end{align*}
which is the strong convexity inequality we aimed to prove.
\end{proof}
\end{document}

%% file: shortcuts.tex
\newcommand{\RR}{\mathbb{R}}
\newcommand{\NN}{\mathbb{N}}

\newcommand{\Cc}{\mathcal{C}}

\newcommand{\Mm}{\mathcal{M}}
\newcommand{\Nn}{\mathcal{N}}

\newcommand{\Pp}{\mathcal{P}}

\newcommand{\Xx}{\mathcal{X}}

\newcommand{\E}{\mathbf{E}}

\renewcommand{\d}{\mathrm{d}}

\newcommand{\TV}{\mathrm{TV}}

\DeclareMathOperator{\KL}{KL}
\DeclareMathOperator{\LSE}{LSE}

\newtheorem{theorem}{Theorem}[section]
\newtheorem{proposition}[theorem]{Proposition}
\newtheorem{lemma}[theorem]{Lemma}
\newtheorem{corollary}[theorem]{Corollary}

\newtheorem{definition}[theorem]{Definition}